\documentclass{amsart}

\usepackage[all,hyperref]{bi-discrete}
\usepackage[backend=biber,maxbibnames=5,maxalphanames=5,style=alphabetic,bibencoding=utf8,giveninits,url=false,isbn=false]{biblatex}
\AtBeginBibliography{\small}

\usepackage[ruled,vlined,norelsize,noend]{algorithm2e}

\usepackage{tikz}
\usetikzlibrary{calc}
\usepackage{pgfplots}
\pgfplotsset{
  width=.65\linewidth,
  axis background/.style={fill=black!5!white},
  grid style={densely dotted,semithick},
  legend style={
    legend columns=1,
    legend pos=outer north east
  },
  compat=1.16,
}
\pgfplotscreateplotcyclelist{MyColors}{%
    {red,mark = *,every mark/.append style={solid,scale=0.4,fill=red}},
	{dotted,red,mark = *,every mark/.append style={solid,scale=0.4,fill=red}},
    {teal,mark = x,every mark/.append style={solid,scale=0.4,fill=teal}},
    {blue,mark = square*,every mark/.append style={solid,scale=0.4,fill=blue}},
    {dotted,teal,mark = x,every mark/.append style={solid,scale=0.4,fill=teal}},
    {dotted,blue,mark = square*,every mark/.append style={solid,scale=0.4,fill=blue}},
{dash dot,red,mark = *,every mark/.append style={solid,scale=0.4,fill=red}},
    {dash dot,teal,mark = x,every mark/.append style={solid,scale=0.4,fill=teal}},
    {dash dot,blue,mark = square*,every mark/.append style={solid,scale=0.4,fill=blue}},}

\tikzset{
  cross/.pic = {
    \draw[rotate = 45] (-#1,0) -- (#1,0);
    \draw[rotate = 45] (0,-#1) -- (0, #1);
  }
}
\usepackage{tikz-3dplot}

\usepackage{mathtools}

\usepackage[ruled,vlined,norelsize,noend]{algorithm2e}

\providecommand{\CBDV}{C_{\textsc{\scriptsize BDV}}}
\providecommand{\gen}{{{\tt gen}_N}}
\providecommand{\genN}{{\tt gen}_N}
\providecommand{\genNsharp}{{\tt gen}^{\sharp}_N}

\providecommand{\level}{{{\tt lvl}}}
\providecommand{\levelN}{{\tt lvl}_N}

\providecommand{\neig}{\texttt{Neigh}}
\providecommand{\gensharp}{{\genNsharp}}
\providecommand{\levelNsharp}{{\level^\sharp}}

\providecommand{\simplex}[1]{\llbracket{#1}\rrbracket}
\providecommand{\maubach}[1]{[{#1}]}
\providecommand{\type}{{{\tt type}_N}}
\providecommand{\typeN}{{\tt type}_N}
\providecommand{\tria}{\mathcal{T}}
\providecommand{\Bisec}{{\tt Bisec}}
\providecommand{\bisec}{{\Bisec}} 
\providecommand{\BisecT}{\Bisec(\tria_0)}
\providecommand{\BisectionClosure}{{\Refine}}
\providecommand{\Refine}{{\tt Refine}}
\providecommand{\mtree}{\mathbb{T}}
\providecommand{\vertices}{\mathcal{V}}
\providecommand{\nodes}{\vertices}
\providecommand{\edges}{\mathcal{E}}
\providecommand{\bse}{\texttt{bse}}
\providecommand{\bsv}{\texttt{bsv}}
\providecommand{\dist}{\textup{dist}}

\begin{document}

\author[L.\ Diening]{Lars Diening}
\author[L.\ Gehring]{Lukas Gehring}
\author[J.\ Storn]{Johannes Storn}

\address[L.\ Diening]{Department of Mathematics, Bielefeld University, Postfach 10 01 31, 33501 Bielefeld, Germany}
\email{lars.diening@uni-bielefeld.de}
\address[L.\ Gehring]{Department of Mathematics, Friedrich-Schiller-Universität Jena, Ernst-Abbe-Platz 2, 07743 Jena, Germany}
\email{lukas.gehring@uni-jena.de}
\address[J.\ Storn]{Faculty of Mathematics \& Computer Science, Institute of Mathematics, Leipzig University, Augustusplatz 10, 04109 Leipzig, Germany}
\email{johannes.storn@uni-leipzig.de}
\thanks{
The work of Lars Diening and Johannes Storn was supported by the Deutsche Forschungsgemeinschaft (DFG, German
Research Foundation) – SFB 1283/2 2021 – 317210226. Lukas Gehring received funding from the European Union's Horizon 2020 research and innovation programme (Grant agreement
No. 891734) and inspiration from the almighty God from whom all ideas issue.}

\subjclass[2020]{
65N50, 
 65Y20. 
}
\keywords{bisection, closure estimate, newest vertex bisection, AFEM, shape regularity}

\title{Adaptive Mesh Refinement for arbitrary initial Triangulations}

\begin{abstract}
  We introduce a simple initialization of the Maubach bisection routine for adaptive mesh refinement which applies to any conforming initial triangulation and terminates in linear time with respect to the number of initial vertices. 
We show that Maubach's routine with this initialization generates meshes that preserve shape regularity and satisfy the closure estimate needed for optimal convergence of adaptive schemes. 
Our ansatz allows for the intrinsic use of existing implementations.
\end{abstract}

\maketitle
\section{Introduction}
\label{sec:introduction}
Adaptive mesh refinements are of uttermost importance for efficient finite element approximations of partial differential equations that exhibit singular solutions.
Starting with the seminal contributions \cite{BinevDahmenDeVore04,Stevenson07} a significant amount of papers investigated and verified the optimal convergence of such schemes, see
\cite{CarstensenFeischlPagePraetorius14} for an overview. 
These schemes base on the so-called adaptive finite element loop. This loop solves the discrete problem, computes local error estimators, marks simplices with large error contributions either by the D\"orfler criterion \cite{Doerfler96} or a maximum marking strategy  \cite{DieningKreuzerStevenson16}, and refines the underlying triangulation locally. 
A key property needed in all optimal convergence results is the closure estimate of the mesh refinement routine displayed in Theorem~\ref{thm:closuregcolored}. This estimate bounds the number of newly created simplices by the accumulated number of marked simplices.
Such estimates have been obtained in \cite{BinevDahmenDeVore04} for two and in \cite{Stevenson08} for higher dimensions for the newest vertex bisection \cite{Mitchell91} and its generalization to higher dimensions by Kossaczký, Maubach, and Traxler \cite{Kossaczky94,Maubach95,Traxler97}. However, the results require an initial condition that is in dimension $n\geq 3$ rather restrictive in the sense that there the initial conditions cannot be satisfied for general triangulations. 
We overcome this drawback by a novel initialization algorithm, while the bisection routine of each single simplex remains the one of~\cite{Maubach95,Traxler97}. Our novel ansatz leads to the following advantages.
\begin{itemize}
\item It applies to any initial triangulation and dimension $n\geq 2$.
\item It preserves shape regularity which is important for the convergence of finite element schemes~\cite{BabAzi76,Oswald15}.
\item It allows to use existing implementations of the routines in \cite{Maubach95,Traxler97} and leads to the same similarity classes of simplices.
\item It avoids additional initial refinements as used in \cite{Kossaczky94,Stevenson08}.
\item The costs of the initialization in Algorithm~\ref{algo:coloring} are linear in the number of initial vertices. 
\item Full uniform refinements ($n$ bisections of each simplex) of the initial triangulation are conforming.
\item It satisfies the closure estimate of Binev--Dahmen--DeVore, see Theorem~\ref{thm:closuregcolored}.
\end{itemize}
Our extension is motivated by one of the author's master thesis \cite{Gehring}. It relies on the observation that any initial triangulation $\tria_0$ in $\mathbb{R}^n$ can be seen as a collection of faces of a triangulation $\tria_0^+$ in $\mathbb{R}^N$ with $n \leq N$ that has suitable initial conditions.
We obtain such a triangulation $\tria_0^+$  by assigning a color to each vertex in $\tria_0$ such that the colors of vertices connected by an edge are different and by adding vertices with the remaining colors to each simplex. 
This extension provides the generation structure exploited in \cite{DieningStornTscherpel23}. This structure applies to subsimplices and so in particular to simplices in $\tria_0$ and their descendants. We use this property to introduce a new notion of generation for simplices. With this notion of generation the closure estimate follows, after overcoming some technical difficulties, by arguments similar to the ones of Binev, Dahmen, DeVore \cite{BinevDahmenDeVore04} and Stevenson \cite{Stevenson08}. We are able to bound the involved equivalence constants in terms of the number of colors $N$, which is limited by the maximal number of initial edges connected to an initial vertex, see Lemma~\ref{lem:largestColor}. 
Notice that the triangulation $\tria_0^+$ is a theoretical tool. Neither do we construct $\tria_0^+$ nor does our refinement routine depend on $\tria_0^+$. In fact, the colors assigned to each simplex provide an order of the vertices which then allows for the application of the bisection routines of Maubach \cite{Maubach95} and Traxler \cite{Traxler97}.
\section{Novel Initialization}
This section introduces our novel initialization for the Maubach routine. 
The definition of this bisection routine uses the notion of simplices.
A $k$-simplex with $0\leq k \leq n$ is the convex hull of $k+1$ affinely independent points~$v_0,\dots, v_k\in \RRn$ and is denoted by $[v_0,\dots,v_k]$. In particular, a 0-simplex is a vertex and a 1-simplex is an edge.
The points $v_0,\dots, v_k$ are called vertices of the $k$-simplex~$S = [v_0,\dots, v_k]$ and we denote the set of all vertices by $\nodes(S) = \lbrace v_0,\dots,v_k\rbrace$. Moreover, $\edges(S) \coloneqq \lbrace [v_i,v_j]\colon i,j=0,\dots,k$ and $i\neq j\rbrace$ denotes the set of all edges in $S$.
If we have a set of simplices $\mathcal{P}$, we denote by $\nodes(\mathcal{P}) \coloneqq \bigcup_{S\in \mathcal{P}} \nodes(S)$ and $\edges(\mathcal{P}) \coloneqq \bigcup_{S\in \mathcal{P}} \edges(S)$ the union of its vertices and edges.
Maubach's routine additionally requires a so-called tag $\gamma \in \lbrace 1,\dots,n\rbrace$ associated to each simplex. 
The tagged simplices are bisected according to Algorithm~\ref{algo:maubach}. The routine determines the bisection edge $\bse(T)$ of a given $n$-simplex $T$ and creates two new simplices containing the midpoint $\bsv(T) \coloneqq \textup{mid}(\bse(T))$ of the bisection edge as new vertex.
\begin{figure}
\begin{algorithm}[H]
  \caption{Bisection of $n$-simplex \cite{Maubach95}}
  \label{algo:maubach}
  \SetAlgoLined%
  \SetKwFunction{FuncBisection}{Bisection}
  \FuncBisection{$T$}{%

    \KwIn{A tagged $n$-simplex~$T=\maubach{v_0, \dots, v_n}_\gamma$ with tag~$\gamma \in \lbrace 1,\dots,n\rbrace$}
    \KwOut{Two children $n$-simplices $T_1$ and $T_2$}    
    
    Set $\gamma' \coloneqq \gamma  - 1$ if $\gamma \geq 2$ and $\gamma' \coloneqq n$ else\tcp*{new tag}
    Set $\bse(T) \coloneqq [v_0,v_\gamma]$\tcp*{bisection edge}
    Set $v' \coloneqq \bsv(T)  \coloneqq \textup{mid}(\bse(T)) \coloneqq  (v_0+v_\gamma)/2$\tcp*{bisection vertex}
    
   	\Return $T_1 \coloneqq \maubach{ v_0, v_1, \dots, v_{\gamma-1}, v', v_{\gamma+1}, \dots,v_n}_{\gamma'}$ and $T_2 \coloneqq \maubach{ v_1,v_2,\ldots, v_\gamma,v', v_{\gamma+1}, \dots,v_n}_{\gamma'}
    $
  }
\end{algorithm}
\end{figure}
This iterative routine leads to the question how to define the order of vertices and the tag in initial simplices. We answer this question by an initialization algorithm that relies on a coloring for vertices in the initial triangulation $\tria_0$ with vertices $\nodes(\tria_0)$. 
\begin{definition}[$N{+}1$-Coloring]
  \label{def:gcoloring}
  We call the pair~$(\tria_0,\frc)$ an \emph{$N{+}1$-colored} triangulation, if $\tria_0$ is a triangulation
   in~$\RRn$ and $\frc\colon \vertices(\tria_0) \to \lbrace 0 ,\dots,N\rbrace$ is a mapping with~$N \geq n$ such that for each simplex $T\in \tria_0$ the colors of its vertices $\nodes(T)$ are distinct.
\end{definition}
With such a coloring we define the order and tag of initial simplices as follows.
\begin{definition}[Maubach initialization]\label{def:MaubachInit}
Let $\tria_0$ be an initial triangulation with $N{+}1$-coloring $\frc$. We sort the vertices $v_0,\dots,v_n$ satisfying $\frc(v_j) < \frc(v_{j+1})$ for all $j = 0,\dots,n-1$ of each initial tagged simplex $T\in \tria_0$ such that
\begin{equation}\label{eq:MaubachSorting}
T = \begin{cases}
[v_n, v_0, v_1,\dots,v_{n-1}]_n&\text{if }\frc(v_n) = N,\\
[v_0,v_1,\dots,v_{n}]_n & \text{else}.
\end{cases}
\end{equation}
\end{definition}
\begin{remark}[Alternative sorting]\label{rem:alternativeSorting}
By redefining the colors $\frc'(v_j) \coloneqq \frc(v_j) + 1$ for $j=0,\dots,N-1$ and $\frc'(v_N) \coloneqq 0$ we can avoid the two cases in \eqref{eq:MaubachSorting} in the sense that we sort the vertices in $T$ according to the color $\frc'(v_j) < \frc(v_{j+1})$ for all $j=0,\dots,n$. Definition~\ref{def:MaubachInit} has the advantage that it is consistent with the notion of coloring in \cite{DieningStornTscherpel23} that provides results needed in our proofs later.
\end{remark}
\begin{remark}[Traxler]\label{rem:Traxler}
The initialization with the equivalent Traxler bisection routine (\cite[Sec.~3]{Traxler97} and \cite[Sec.~2]{Stevenson08}) uses the same sorting of vertices as in \eqref{eq:MaubachSorting} and the simplex type $\gamma = 0$.
\end{remark}
Before we state the resulting routine's properties, let us introduce the notion of conforming triangulations and the bisection routine with closure.
\begin{definition}[Triangulation]
  \label{def:triangulation}
  Let $\tria$ be a collection of closed $n$-simplices $T\subset \RRn$ with pairwise disjoint interior.
Such a partition $\tria$ is called \emph{(conforming) triangulation}, 
if the intersection of any two $n$-simplices~$T_1,T_2 \in \tria$ is either empty or an $m$-subsimplex of both~$T_1$ and $T_2$ with $m \in \set{0,\dots, n}$.
\end{definition}
For a regular triangulation $\tria_0$  the colors of all vertices of each $T\in\tria_0$ are distinct, if and only if the colors of the vertices of each edge $e\in\edges(\tria_0)$ are distinct. This motivates Algorithm~\ref{algo:coloring} which assigns a color to each vertex in $\tria_0$ such that the assumption in Definition~\ref{def:gcoloring} is satisfied. 
The complexity of Algorithm~\ref{algo:coloring}  is linear in the number of initial vertices and the following lemma shows that the resulting number $N$ of colors is limited by the maximal degree of a vertex in $\tria_0$, that is, the number of edges sharing the same vertex
$
\# \edges_0(v) \coloneqq \# \lbrace e \in \edges(\tria_0)\colon v\in e\rbrace$ with $v\in \vertices(\tria_0)$.
\begin{lemma}[Largest color]\label{lem:largestColor}
Let the coloring $\frc$ result from Algorithm~\ref{algo:coloring}. Then the
maximal degree bounds the largest color in the sense that
\begin{align*}
N  \coloneqq \max_{v\in \vertices(\tria_0)} \frc(v)  \leq \max_{v\in \vertices(\tria_0)} \# \edges_0(v).
\end{align*}
\end{lemma}
\begin{proof}
It follows by induction that the smallest number in $\mathbb N_0\setminus \left\{\frc(w)~\middle|~[v,w]\in\edges(\tria_0)\right\}$ is at most $\#\edges_0(v)$. This observation yields the lemma.
\end{proof}
\begin{figure}
\begin{algorithm}[H]
  \label{algo:coloring}
  \caption{Generalized coloring with $N{+}1$ colors}
  \SetAlgoLined%
  \SetKwFunction{FuncBisection}{SetColor}
  \FuncBisection{$\tria_0$}{%
    
    \KwIn{Conforming triangulation $\tria_0$}
    
    \KwOut{Generalized coloring $\frc\colon \vertices(\tria_0)\to \lbrace 0 ,\dots, N\rbrace$}

	Set $\frc(v) \coloneqq \infty$ for all $v\in \vertices(\tria_0)$\;
	\ForEach{$v\in \vertices(\tria_0)$}
	{Set $\frc(v) \coloneqq \min\left(\mathbb{N}_0\setminus\lbrace \frc(w)\colon [v,w]\in\edges(\tria_0) \rbrace\right)$
	\,\tcp*{smallest color not already attained by a neighboring vertex}}
	Set $N\coloneqq \max \lbrace \frc(v)\colon v\in \vertices(\tria_0)\rbrace$\;  
  }
\end{algorithm}
\end{figure}
%

In order to guarantee that bisections of simplices in conforming triangulations do not create hanging vertices, it is necessary to apply the conforming closure displayed in Algorithm~\ref{algo:closure-recursive}.
\begin{figure}
\begin{algorithm}[H]
  \caption{Bisection with conforming closure (recursive)}\label{algo:closure-recursive}
  \SetAlgoLined%
  \SetKwFunction{FuncBisect}{Refine}
  \FuncBisect{$\tria, T$}{
    
    \KwIn{Conforming triangulation $\tria$ and an $n$-simplex~$T\in \tria$}
    \KwOut{Coarsest conforming refinement~$\tria'$ of $\tria$ where $T$ is bisected}
    
    Set $e\coloneqq\bse(T)$ and $\omega_\tria(e) \coloneqq \{S\in \tria~|~e\in\edges(S)\}$\;

    \If{there exists a $T'\in\omega_\tria(e)$ with $\bse(T')\neq e$}{\Return \FuncBisect{\FuncBisect{$\tria,T'$},$T$}\tcp*{flag $T'$ for refinement}}
    
    \Else{\Return $\tria\setminus\omega_\tria(e)\cup\bigcup_{T'\in \omega_\tria(e)}\texttt{Bisection}(T')$\tcp*{conforming mesh}}
  }  
\end{algorithm}
\end{figure}
%
Notice that there are non-recursive formulations that, in contrast to Algorithm~\ref{algo:closure-recursive}, require weaker initial conditions in order to terminate.
However, the recursive formulation has analytical advantages which we exploit in Lemma \ref{lem:RefChainsNew} below. 
If the recursive algorithm terminates, these closure routines are equivalent. 

Let $\BisecT = \Bisec(\tria_0,\frc)$ denote the set of all triangulations that can be obtained by successive applications of Algorithm~\ref{algo:closure-recursive} to some $N{+}1$-colored initial triangulation $(\tria_0,\frc)$ with initialization as introduced in Definition~\ref{def:MaubachInit}. We denote the set of possible simplices by $\mtree \coloneqq \bigcup \BisecT$.
Let $d(T)$ and $D(T)$ denote the diameters of the largest inscribed and smallest ball including $T\in \mtree$, respectively. We define the shape-regularity of a simplex $T\in \mtree$ by 
 \begin{equation}\label{eq:DefShapeReg}
\gamma(T) \coloneqq D(T)/d(T) > 1.
\end{equation}
\begin{theorem}[Basic properties]
  \label{thm:basic-properties}
  Let $(\tria_0,\frc)$ be an $N{+}1$-colored initial triangulation with initialization as in Definition~\ref{def:MaubachInit}.
  \begin{enumerate}
  \item \label{itm:basic-colored-terminate}%
    The bisection algorithm defined by Algorithms~\ref{algo:maubach} and~\ref{algo:closure-recursive} terminates for each $\tria\in \BisecT$ and marked simplex $T\in \tria$.
  \item \label{itm:basic-colored-shape}%
There are at most $n!\, n\, 2^{n-2}$ classes of similar simplices in $\mtree$ for each $T_0 \in \tria_0$. Moreover, the simplices in $\mtree$ are shape-regular in the sense that their shape regularity~$\gamma$ defined in \eqref{eq:DefShapeReg} satisfies with constant $C_{sr} \coloneqq 2n (n+\sqrt{2}-1 )$
\begin{align*}
\gamma(T) \leq C_{sr} \gamma(T_0) \qquad\text{for all }T\in \mtree\text{ with ancestor }T_0 \in \tria_0.
\end{align*}    
  \item \label{itm:basic-colored-uniform}
    Consecutive uniform full refinements of $\tria_0$, that are $n$ successive bisections of all simplices, are conforming.
  \item 
    \label{itm:basic-colored-lattice} $(\BisecT, \vee, \wedge)$ is a distributive ordered lattice, where $\tria_1 \vee \tria_2$ is the coarsest common refinement and $\tria_1 \wedge \tria_2$ is the finest common coarsening of~$\tria_1$ and $\tria_2$.
  \end{enumerate}
\end{theorem}
We verify these properties in the following section (and the upper bound in \ref{itm:basic-colored-shape} in the appendix). 
Moreover, we verify a closure estimate which reads as follows.

The AFEM loop \texttt{Solve}--\texttt{Estimate}--\texttt{Mark}--\texttt{Refine} generates a sequence of triangulations~$\mathcal{T}_k \in \BisecT$ in the following iterative way. After calculating the finite element solution solution with underlying triangulation~$\mathcal{T}_k$, error indicators lead to a set~$\mathcal{M}_k \subset \mathcal{T}_k$ of marked $n$-simplices. A bisection of all those $n$-simplices with conformal closure generates a new triangulation~$\mathcal{T}_{k+1}$, that is,
\begin{align*}
\tria_{k+1} = \BisectionClosure(\tria_k,\mathcal M_k) \coloneqq \bigvee_{T\in \mathcal{M}_k} \BisectionClosure(\tria_k,T).
\end{align*}
The proofs of optimal convergence (see \cite{CarstensenFeischlPagePraetorius14} for an axiomatic approach) require the control of the effect of the conforming closure in Algorithm~\ref{algo:closure-recursive}, displayed in the following Theorem~\ref{thm:closuregcolored}. The result has been proven for $n=2$ in \cite{BinevDahmenDeVore04} and $n \geq 2$ in~\cite{Stevenson08} under additional assumptions on the initial triangulation, see the end of Section~\ref{subsec:gener-color-init} for a discussion.
\begin{theorem}[Closure estimate]
  \label{thm:closuregcolored}
  Let $\tria_0$ be an $N{+}1$-colored triangulation, let $\mathcal{M}_\ell\subset \tria_\ell$ denote the sets of marked elements in an AFEM loop with $\mathcal{T}_{\ell+1} = \BisectionClosure(\tria_\ell,\mathcal{M}_\ell)$. 
  There exists a constant $\CBDV < \infty$ with
  \begin{equation}\label{eq:ClosureEstThm}
      \#\tria_L - \# \tria_0 \leq \CBDV\, \sum_{\ell=0}^{L-1} \# \mathcal{M}_\ell\qquad\text{for all }L \in \mathbb{N}.
  \end{equation}
We have $\CBDV \leq C N^{n}$, where $C<\infty$ is a constant that depends on $n$, the shape regularity of $\tria_0$, and the quasi-uniformity $\max_{T,T'\in \tria_0} |T|/|T'|$ of $\tria_0$, but not on $N$.
\end{theorem}
Notice that both results, Theorem~\ref{thm:basic-properties} and \ref{thm:closuregcolored}, are known for ``suitable'' initial triangulations $\tria_0$. The main advantage of our approach is that these properties extend to any initial triangulation.
\section{Theoretical Investigation of the novel Initialization}
The key idea to extend the results in Theorem~\ref{thm:basic-properties} and \ref{thm:closuregcolored} from ``suitable'' triangulations to arbitrary initial triangulations is the following. We relate our initial $N{+}1$-colored triangulation $(\tria_0,\frc)$ in $\mathbb{R}^n$ to some triangulation $\tria_0^+$ of some higher-dimensional domain in $\mathbb{R}^N$ with ``suitable'' initial conditions. We then use an established structure given by the higher-dimensional triangulation $\tria_0^+$ to define some structure for triangulations in $\BisecT$ which finally leads to the statements in Theorem~\ref{thm:basic-properties} and \ref{thm:closuregcolored}. Section~\ref{subsec:color-init-triang} introduces and discusses the structure for such ``suitable'' initial conditions. 
Thereby, we introduce and motivate the notions of generation, level, and type needed in the following proofs and discuss relations to existing results.
Section~\ref{subsec:gener-color-init} extends this structure to $(\tria_0,\frc)$ with $N{+}1$-coloring $\frc$.
We use this structure in Section~\ref{subsec:ClosureEstimate} to verify the theoretical results.
\subsection{Initial Coloring with $N = n$}
\label{subsec:color-init-triang}
Typical assumptions on $\tria_0$ in the literature like Traxler's \emph{reflected domain partition} condition~\cite[Sec.~6]{Traxler97} can be rewritten as an $n{+}1$-coloring, that is, $N=n$ in Definition~\ref{def:gcoloring}.
In the remainder of this subsection we assume that~$(\tria_0,\frc)$ has such a coloring $\frc$ with $N{+}1=n{+}1$ colors, 
that is, the vertices in
each initial simplex $T = \maubach{v_N,v_0,\dots,v_{N-1}}_N \in \tria_0$ satisfy
\begin{equation}
  \label{eq:maubach-color}
  \frc (v_j) = j\qquad\text{for all }j=0,\dots,N.
\end{equation}
Under this assumption it is known (cf.~\cite[Sec.~4]{Traxler97}) that consecutive uniform refinements of $\tria_0$ (bisections of each $N$-simplex) are conforming. 
The property that consecutive uniform refinements are conforming is equivalent to Stevenson's matching neighbor condition \cite[Thm.~4.3 and Rem.~4.4]{Stevenson08}. For Stevenson's more general assumption Theorem~\ref{thm:basic-properties} is well known. In particular, the properties~\ref{itm:basic-colored-terminate} and \ref{itm:basic-colored-uniform} can be found in~\cite{Maubach95,Traxler97,Stevenson08} and~\ref{itm:basic-colored-lattice} in \cite{DieningKreuzerStevenson16,DieningStornTscherpel23}.
The statement about the similarity classes in \ref{itm:basic-colored-shape} can be found in \cite[Thm.~4.5]{AMP.2000}. We verify the shape regularity result with explicit constant $C_{sr}$ in \ref{itm:basic-colored-shape} in the appendix.
Moreover, the closure estimate in Theorem~\ref{thm:closuregcolored} has been shown for $N=2$ in \cite{BinevDahmenDeVore04} and $N \geq 2$ in~\cite{Stevenson08}. 
Additionally, the restrictive initial conditions allow for a mesh-grading result, see \cite[Thm.~1.2]{DieningStornTscherpel23}.
The proofs in \cite{DieningStornTscherpel23} exploit some fine properties of~$\BisecT$.
Since these properties are important for our proofs in Section~\ref{subsec:ClosureEstimate}, we introduce and explain them in the following. 

It is quite standard, e.g.~\cite{BinevDahmenDeVore04}, to assign to each $N$-simplex~$T \in \mtree$ a generation. For each $T \in \tria_0$ we set~$\gen(T) \coloneqq 0$. If the bisection of~$T$ creates children~$T_1$ and $T_2$, then $\gen(T_1) \coloneqq \gen(T_2) \coloneqq \gen(T)+1$. However, it was shown in~\cite{DieningStornTscherpel23} that for $N{+}1$-colored initial partitions~$(\tria_0,\frc)$ with $N=n$ it is even possible to assign a generation to each vertex~$v \in \vertices(\mtree)$ that is compatible with the generation of $N$-simplices in the sense
\begin{equation}
  \label{eq:GenMaxVertices}
  \gen(T) = \max_{v\in \vertices(T)} \gen(v)\qquad\text{for all } T \in \mtree.
\end{equation}
For this they defined $\gen(v) \coloneqq -\frc(v)$ for each initial vertex $v \in \vertices(\tria_0)$, where $\frc$ is the color map of~$\tria_0$. Then the generation of each bisection vertex $\bsv(T)$ resulting from the bisection of a simplex $T\in \mtree$ is set to $\gen(\bsv(T)) \coloneqq \gen(T) +1$. 
Since for colored initial partitions uniform refinements are conforming \cite{Traxler97}, this definition is well-posed. An induction reveals that the generations of vertices within a simplex are distinct.
Thus, we can sort the vertices of each $m$-subsimplex $S = [v_0,\dots,v_m]$ with $m\leq N$ by decreasing generations. For this we use the notation
\begin{equation}
  \label{eq:sorted}
  S = \simplex{v_0,\dots, v_m}\qquad\text{whenever } \gen(v_0) > \gen(v_1) > \dots > \gen(v_m)
\end{equation}
and assign, using the formula in~\eqref{eq:GenMaxVertices}, the generation
\begin{equation}
  \label{eq:GenMaxVerticesSub}
  \gen(S) \coloneqq \gen(v_0) = \max_{v\in \vertices(S)} \gen(v).
\end{equation}
Based on the generations, \cite[Sec.~3.3]{DieningStornTscherpel23} derives the bisection rule in Algorithm~\ref{algo:subsimplex} that agrees with the bisection rule of Algorithm~\ref{algo:maubach} and can additionally be applied to $m$-subsimplices with $m = 1,\dots, N$ in the sense that it generates $m$-subsimplices of descendants of an initial $N$-simplex.
The algorithm involves the notion of levels~$\levelN(\bigcdot) \in \setZ$ and types~$\type(\bigcdot) \in \set{1,\dots, N}$, which are defined for vertices and $m$-simplices by
\begin{equation}
  \label{eq:leveltype}
  \gen(\bigcdot) = N (\levelN(\bigcdot)-1) + \type(\bigcdot).
\end{equation}
\begin{figure}
\begin{algorithm}[H]
  \label{algo:subsimplex}
  \caption{Bisection of $m$-subsimplex in dimension~$N$}
  \SetAlgoLined%
  \SetKwFunction{FuncBisection}{Bisection}
  \FuncBisection{$S$}{%
    
    \KwIn{An $m$-simplex $S=\simplex{v_0, \dots, v_m}$ sorted by decreasing vertex generation with $1\leq m \leq N$}
   
    \KwOut{Bisection edge $\bse(S)$ as well as generation of the bisection vertex~$\bsv(S) = \textup{mid}(\bse(S))$}

    \eIf{$\levelN(v_m) \neq \levelN(v_{m-1})$}{
      $\bse(S) =\simplex{v_{m-1}, v_m}$\tcp*{Two oldest vertices}

      $\gen(\bsv(S)) = \gen(v_{m-1}) + N$\;
    }{
      $j \coloneqq \min \set{k\colon \levelN(v_k) = \levelN(v_m)}$\;
      
      $\bse(S)=\simplex{v_j,v_m}$\tcp*{Youngest and oldest vertex of old level}

      $\gen(\bsv(S)) = \gen(v_m) + 2N+1 - \type(v_j)$\;
    }
  }
\end{algorithm}
\end{figure}%

We conclude this subsection with commenting on the restrictions of the assumption of a $n{+}1$-colorable initial triangulations.
\begin{remark}[Colorability]
  \label{rem:colorability}
Not every initial triangulation~$\tria_0$ can be $n{+}1$-colored.  A necessary condition for $n=2$ is that every interior vertex has an even number of simplices sharing this vertex, see Figure~\ref{fig:noncolorable}. 
\begin{figure}
  \centering
  \begin{minipage}{.45\textwidth}
    \centering
    \tdplotsetmaincoords{62}{33}
\begin{tikzpicture}[scale = 2,tdplot_main_coords]
\fill (0,0,0) circle (0.3pt) node[above] {$0$};
\fill (1.0,0.0,0.0) circle (0.3pt) node[right] {$2$};
\fill (0.30901699437494745,0.9510565162951535,0.0) circle (0.3pt) node[right] {$1$};
\fill (-0.8090169943749473,0.5877852522924731,0.0) circle (0.3pt) node[above] {$2$};
\fill (-0.8090169943749473,-0.5877852522924731,0.0) circle (0.3pt) node[left] {$2$};
\fill (0.3090169943749473,-0.9510565162951535,0.0) circle (0.3pt) node[below] {$1$};
\fill (-1.1,0.25) circle (0.3pt) node[left] {$1$};

\draw (0,0,0) -- (0,0,0);
\draw (0,0,0) -- (1.0,0.0,0.0);
\draw (0,0,0) -- (0.30901699437494745,0.9510565162951535,0.0);
\draw (0,0,0) -- (-0.8090169943749473,0.5877852522924731,0.0);
\draw (0,0,0) -- (-1.1,0.25);
\draw (0,0,0) -- (-0.8090169943749475,-0.587785252292473,0.0);
\draw (0,0,0) -- (0.3090169943749473,-0.9510565162951535,0.0);
\draw (1.0,0.0,0.0) -- (0.30901699437494745,0.9510565162951535,0.0);
\draw (1.0,0.0,0.0) -- (0.3090169943749473,-0.9510565162951535,0.0);
\draw (0.30901699437494745,0.9510565162951535,0.0) -- (-0.8090169943749473,0.5877852522924731,0.0);
\draw (-0.8090169943749473,0.5877852522924731,0.0) -- (-1.1,0.25);
\draw (-1.1,0.25) -- (-0.8090169943749475,-0.587785252292473,0.0);
\draw (-0.8090169943749475,-0.587785252292473,0.0) -- (0.3090169943749473,-0.9510565162951535,0.0);
\end{tikzpicture}

  \end{minipage}
  \begin{minipage}{.45\textwidth}
    \centering
    \tdplotsetmaincoords{62}{33}
\begin{tikzpicture}[scale = 2,tdplot_main_coords]
\fill (0,0,0) circle (0.3pt) node[above] {$0$};
\fill (1.0,0.0,0.0) circle (0.3pt) node[right] {$2$};
\fill (0.30901699437494745,0.9510565162951535,0.0) circle (0.3pt) node[right] {$1$};
\fill (-0.8090169943749473,0.5877852522924731,0.0) circle (0.3pt) node[above] {$2$};
\fill (-0.8090169943749473,-0.5877852522924731,0.0) circle (0.3pt) node[left] {\textup{?}};
\fill (0.3090169943749473,-0.9510565162951535,0.0) circle (0.3pt) node[below] {$1$};

\draw (0,0,0) -- (0,0,0);
\draw (0,0,0) -- (1.0,0.0,0.0);
\draw (0,0,0) -- (0.30901699437494745,0.9510565162951535,0.0);
\draw (0,0,0) -- (-0.8090169943749473,0.5877852522924731,0.0);
\draw (0,0,0) -- (-0.8090169943749475,-0.587785252292473,0.0);
\draw (0,0,0) -- (0.3090169943749473,-0.9510565162951535,0.0);
\draw (1.0,0.0,0.0) -- (0.30901699437494745,0.9510565162951535,0.0);
\draw (1.0,0.0,0.0) -- (0.3090169943749473,-0.9510565162951535,0.0);
\draw (0.30901699437494745,0.9510565162951535,0.0) -- (-0.8090169943749473,0.5877852522924731,0.0);
\draw (-0.8090169943749473,0.5877852522924731,0.0) -- (-0.8090169943749475,-0.587785252292473,0.0);
\draw (-0.8090169943749475,-0.587785252292473,0.0) -- (0.3090169943749473,-0.9510565162951535,0.0);
\end{tikzpicture}  
  \end{minipage}
  \vspace*{-1cm}
  
  \caption{Colorable triangulation (left) and a non-colorable triangulation (right) for $n=2$}
  \label{fig:noncolorable}
\end{figure}
\end{remark}
Due to the restrictive requirements for $n{+}1$-colorable initial triangulations, we relax this assumption in Section~\ref{subsec:gener-color-init}. Before, we discuss existing alternatives.

The lack of $n{+}1$-colorability stated in Remark~\ref{rem:colorability} for many initial triangulations is a severe drawback. 
For $n=2$ this problem can be overcome by choosing a bisection edge for each triangle such that the property to be a bisection edge is global, i.e.\ independent of the triangle, see \cite{BinevDahmenDeVore04}. This is possible for any triangulation with~$n=2$ using a perfect matching for cubic graphs; an efficient algorithm to find such a matching was given in \cite{BiedlBoseDemaineLubiw01}. The bisection rule of Algorithm~\ref{algo:maubach} is also applicable in this situation. The resulting triangulations have the basic properties of Theorem~\ref{thm:basic-properties} and satisfy the closure estimate in Theorem~\ref{thm:closuregcolored}.
Moreover, \cite{KarkulikPavlicekPraetorius13} verifies Theorem~\ref{thm:closuregcolored} for arbitrary initializations in $n=2$ showing that, unlike in higher dimensions, an initial coloring for the newest vertex bisection is not needed.\label{page:n2SpecialCase} 
  
Stevenson replaces the coloring condition for $n\geq 2$ by the weaker \emph{matching neighbor} condition in \cite{Stevenson08}. 
The resulting triangulations have the basic properties of Theorem~\ref{thm:basic-properties} and those of Theorem~\ref{thm:closuregcolored}, too. However, not every triangulation can be initialized with tagged simplices satisfying this condition for~$n \geq 3$. It is for example necessary that each interior bisection edge is surrounded by an even number of tetrahedrons \cite[Lem.~1.7.14]{Schoen}. This is not possible in all situations. Thus, this condition has for~$n\geq 3$ the same problems as the coloring condition.
Kossaczký and Stevenson remedied this problem by an initial refinement \cite{Kossaczky94, Stevenson08}, bisecting every initial simplex into $(n{+}1)!/2$ simplices, each one containing a whole original edge before coloring. However, this approach worsens the shape regularity and increases the size of the resulting novel initial triangulation.

B\"ansch~\cite{Baensch91}, Arnold, Mukherjee and Pouly~\cite{AMP.2000} and Schön~\cite{Schoen} suggested to introduce five types of tetrahedrons ($P_u$, $P_f$, $A$, $M$ and $O$ in the language of \cite{AMP.2000}) with special bisection rules. Only three of them ($P_u$, $P_f$ and $A$) appear in the context of colored triangulations with the bisection algorithm of Maubach and Traxler. The other two types $(M$ and $O$) are used only for some tetrahedrons of the initial triangulations. Bisecting those two additional types creates tetrahedrons of type $P_u$, $P_f$ and $A$, whose bisection then follows the rules of Maubach and Traxler. 
Consequently, the resulting triangulations are again shape-regular and form a distributive ordered lattice.
However, consecutive uniform refinements are no longer conforming, cf.~\ref{itm:basic-colored-uniform} of Theorem~\ref{thm:basic-properties}. 
Moreover, it is unknown if triangulations resulting from the strategies in \cite{Baensch91,AMP.2000} satisfy Theorem~\ref{thm:closuregcolored}.
If the closure estimate in \eqref{eq:ClosureEstThm} is satisfied for these strategies, the constant $\CBDV$ depends additionally on the number of initial simplices, since the refinement strategy bisects the longest edge of any initial simplex, see Figure~\ref{fig:BadLongestEdge} for an illustration. 
The same drawback experiences the alternative algorithm in~\cite{Schoen} which satisfies \eqref{eq:ClosureEstThm} with constant $\CBDV$ depending on the number of simplices in $\tria_0$, see \cite[Thm.~1.8.23]{Schoen}, and the newest vertex bisection for $n=2$ investigated in \cite[Thm.~2 and Lem.~7]{KarkulikPavlicekPraetorius13} with $\CBDV$ depending on a constant $C_\textup{chain}$ bounded from below by the longest refinement chain in $\tria_0$. 
A similar drawback must apply to longest edge bisection schemes \cite{Rivara84,Rivara91}, where currently no theoretical results on closure estimates and preservation of shape regularity in higher dimensions exist.
%
  \begin{figure}[h!]
    \begin{tikzpicture}
    \draw (0,0)       -- (9.5,0);
    \draw (0,0)       -- (0.5,0.866);
    \draw (0.5,0.866) -- (11,0.866);
    \draw (1.1,0)       -- (0.5,0.866) node [midway, left=-0.1cm] {\tiny$\, \rightarrow$};
    \draw (1.8,0.866) -- (1.1,0)node [midway, left=-0.1cm] {\tiny$\, \rightarrow$};
    \draw (2.6,0)       -- (1.8,0.866) node [midway, left=-0.1cm] {\tiny$\, \rightarrow$};
    \draw (3.5,0.866) -- (2.6,0)node [midway, left=-0.1cm] {\tiny$\, \rightarrow$};
    \draw (4.5,0)       -- (3.5,0.866)node [midway, left=-0.1cm] {\tiny$\, \rightarrow$};
    \draw (5.6,0.866) -- (4.5,0)node [midway, left=-0.1cm] {\tiny$\, \rightarrow$};
    \draw (6.8,0)       -- (5.6,0.866)node [midway, left=-0.1cm] {\tiny$\, \rightarrow$};
    \draw (8.1,0.866) -- (6.8,0)node [midway, left=-0.1cm] {\tiny$\, \rightarrow$};
    \draw (9.5,0)       -- (8.1,0.866)node [midway, left=-0.1cm] {\tiny$\, \rightarrow$};
    \draw (11,0.866) -- (9.5,0)node [midway, left=-0.1cm] {\tiny$\, \rightarrow$};
  \end{tikzpicture}
  \caption{Initial triangulation with bisection edges (marked by arrows) where the refinement of the simplex on the left causes the refinement of all simplices.}\label{fig:BadLongestEdge}
  \end{figure}%
\subsection{Initial Coloring with $N \geq n$}
\label{subsec:gener-color-init}
Remark~\ref{rem:colorability} shows that not every triangulation~$\tria_0$ in $\RRn$ can be $n{+}1$-colored.
We overcome this difficulty in Definition~\ref{def:gcoloring} by allowing for more colors $0,\dots,N$ with $N \geq n$. 
Such a coloring can always be obtained, see Algorithm~\ref{algo:coloring}.
To derive a proper bisection algorithm for any $N{+}1$-colored~$\tria_0$, we think of each~$T \in \tria_0$ as an $n$-subsimplex of a virtual $N$-simplex in~$\RRN$ by embedding~$\tria_0$ into~$\RRN$ and adding virtual vertices to each $n$-simplex so that it becomes an $N$-simplex. These virtual $N$-simplices are only connected via their $n$-subsimplices given by~$\tria_0$. 
The additional $N{-}n$ vertices are colored with the remaining $N{-}n$ colors,
see Figure~\ref{fig:noncolorable-lifted}. We then apply the subsimplex bisection rule of Algorithm~\ref{algo:subsimplex} with~$m=n$. 

Note that the previous subsection provides an introduction, explaining the origins of the notions $\gen$, $\level_N$, and $\type$. Going forward, we only require Algorithm~\ref{algo:subsimplex}, which defines the generation $\gen$ of bisection vertices. Once the generation of a vertex is known, the relation in \eqref{eq:leveltype} determines the level $\level_N$ and type $\type$. All of these quantities serve as theoretical tools, since the resulting bisection routine is equivalent to Maubach's original routine, as demonstrated by the following theorem.
\begin{theorem}[Equivalent refinements]\label{thm:EquiRefinements}
Let $\tria_0$ be an $N{+}1$-colored initial triangulation. Successive bisections of simplices in $\tria_0$ according to Algorithm~\ref{algo:subsimplex} lead to the same simplices as successive applications of Algorithm~\ref{algo:maubach} to simplices in $\tria_0$ with the initialization in Definition~\ref{def:MaubachInit}.
\end{theorem}
\begin{proof}
We verify this theorem inductively by claiming that bisections with Algorithm~\ref{algo:maubach} or \ref{algo:subsimplex} lead to the same children. Additionally, we show that each resulting tagged $n$-simplex $T = [v_0,\dots,v_n]_\gamma$ with $\gamma\in \lbrace 1,\dots,n\rbrace$ satisfies the following:
\begin{enumerate}
\item $\level_N(v_{\gamma+1}) = \dots = \level_N(v_{n})\text{ and }\gen(v_{\gamma+1}) > \dots > \gen(v_{n})$ if $\gamma < n$,\label{itm:Proof1}\\[-.5em]
\item $\level_N(v_1) = \dots = \level_N(v_\gamma) \text{ and } \gen(v_1) > \dots > \gen(v_\gamma)$,\label{itm:Proof2}\\[-.5em]
\item \label{itm:Proof4}$\begin{cases}
\gen(v_0) > \gen(v_1)\quad \text{and}\quad \level_N(v_0) = \level(v_1)&\text{or}\\
\gen(v_0) < \gen(v_1)\quad \text{and}\quad \level_N(v_0) < \level(v_1),
\end{cases}$
\item\label{itm:Proof5}
for $\gamma < n$ one has $\level_N(v_\gamma) + 1 = \level_N(v_{\gamma+1})$ and\\
$\begin{cases}
\gen(v_{\gamma+1}) < \gen(v_{\gamma})+2N+1-\type(v_0)& \text{if }\level_N(v_0) = \level_N(v_1),\\
\gen(v_{\gamma+1}) < \gen(v_{\gamma})+N& \text{if }\level_N(v_0) < \level_N(v_1).
\end{cases}$
\end{enumerate}

Indeed, the initialization in Definition~\ref{def:MaubachInit} and the definition $\gen(v_j) \coloneqq -\frc(v_j)$ after \eqref{eq:GenMaxVertices} guarantee this property for any $T \in \tria_0$, verifying the base case.

If the properties in \ref{itm:Proof1}--\ref{itm:Proof5} are satisfied, the sorted simplex~\eqref{eq:sorted} reads
\begin{equation*}
T= \begin{cases} 
\simplex{v_{\gamma + 1},\dots,v_n, v_0,\dots,v_\gamma }&\text{if } \level_N(v_0) = \level(v_1), \\
\simplex{v_{\gamma + 1},\dots,v_n, v_1,\dots,v_\gamma, v_0 } &\text{if }\level_N(v_0) < \level(v_1).
\end{cases}
\end{equation*}
Hence, the definitions of the algorithms show that they lead to the same bisection vertex $\bsv(T) = v' = \textup{mid}([v_0,v_\gamma])$ and corresponding children
\begin{equation*}
\begin{aligned}
T_1 &= \maubach{ v_0, v_1, \dots, v_{\gamma-1}, v', v_{\gamma+1}, \dots,v_n}_{\gamma'},\\
T_2 &= \maubach{ v_1,v_2,\ldots, v_\gamma,v', v_{\gamma+1}, \dots,v_n}_{\gamma'}
\end{aligned}
\end{equation*}
with tag $\gamma' = \gamma -1$ for $\gamma \geq 2$ and $\gamma' = n$ for $\gamma = 1$.
It remains to show that the tagged simplices $T_1$ and $T_2$ satisfy \ref{itm:Proof1}--\ref{itm:Proof5}.
\begin{itemize}[left=0pt]
\item Case 1 ($\level(v_0) = \level(v_1)$). Algorithm~\ref{algo:subsimplex} and \eqref{eq:leveltype} imply that
\begin{equation*}
\begin{aligned}
\gen(v') & = \gen(v_\gamma) + 2N + 1 - \type(v_0)\\
& = N (\level_N(v_\gamma) + 1) + 1 - ( \type(v_0) - \type(v_\gamma)).
\end{aligned}
\end{equation*}
Since $\type(v_0) - \type(v_\gamma) \in [1,N-1]$, we obtain $\level_N(v') = \level(v_\gamma) + 1$. Combining this identity with the induction hypothesis verifies \ref{itm:Proof1}--\ref{itm:Proof5}.

\item Case 2 ($\level(v_0) < \level(v_1)$). By Algorithm~\ref{algo:subsimplex} we obtain $\level_N(v') = \level_N(v_\gamma) + 1$. This identity and the induction hypothesis verify \ref{itm:Proof1}--\ref{itm:Proof5}.\qedhere
\end{itemize}
\end{proof}
The theorem shows that the coloring is solely needed for the initialization and that the generation structure is a purely theoretical tool not needed in implementations. After that, the refinement uses the established routine of Maubach and Traxler. 
The virtual extension is only a theoretical concept and the original idea of our method. 
We emphasize that we do not need to extend~$\tria_0$ in any practical computation. Moreover, the extension to virtual $N$-simplices has limitations discussed in Remark~\ref{rem:problem-virtual}. 
\begin{remark}[Virtual extension]
  \label{rem:virtual-extension}
  Figure~\ref{fig:noncolorable-lifted} depicts how a not $3$-colorable triangulation like Figure~\ref{fig:noncolorable} (right) can be $4$-colored. It also illustrates that~$\tria_0$ can be interpreted as the collection of $2$-subsimplices in $\mathbb{R}^3$.
\end{remark}

\begin{figure}
  \centering
  \begin{minipage}{.45\textwidth}
    \centering
    \tdplotsetmaincoords{62}{33}
\begin{tikzpicture}[scale = 2,tdplot_main_coords]
\fill (0,0,0) circle (0.3pt) node[above] {$0$};
\fill (1.0,0.0,0.0) circle (0.3pt) node[right] {$2$};
\fill (0.30901699437494745,0.9510565162951535,0.0) circle (0.3pt) node[right] {$1$};
\fill (-0.8090169943749473,0.5877852522924731,0.0) circle (0.3pt) node[above] {$2$};
\fill (-0.8090169943749473,-0.5877852522924731,0.0) circle (0.3pt) node[left] {?};
\fill (0.3090169943749473,-0.9510565162951535,0.0) circle (0.3pt) node[below] {$1$};

\draw (0,0,0) -- (0,0,0);
\draw (0,0,0) -- (1.0,0.0,0.0);
\draw (0,0,0) -- (0.30901699437494745,0.9510565162951535,0.0);
\draw (0,0,0) -- (-0.8090169943749473,0.5877852522924731,0.0);
\draw (0,0,0) -- (-0.8090169943749475,-0.587785252292473,0.0);
\draw (0,0,0) -- (0.3090169943749473,-0.9510565162951535,0.0);
\draw (1.0,0.0,0.0) -- (0.30901699437494745,0.9510565162951535,0.0);
\draw (1.0,0.0,0.0) -- (0.3090169943749473,-0.9510565162951535,0.0);
\draw (0.30901699437494745,0.9510565162951535,0.0) -- (-0.8090169943749473,0.5877852522924731,0.0);
\draw (-0.8090169943749473,0.5877852522924731,0.0) -- (-0.8090169943749475,-0.587785252292473,0.0);
\draw (-0.8090169943749475,-0.587785252292473,0.0) -- (0.3090169943749473,-0.9510565162951535,0.0);
\end{tikzpicture}  
  \end{minipage}
  \begin{minipage}{.45\textwidth}
    \centering
    \tdplotsetmaincoords{62}{33}
\begin{tikzpicture}[scale = 2,tdplot_main_coords]
\fill (0,0,0) circle (0.3pt) node[above] {$0$};
\fill (1.0,0.0,0.0) circle (0.3pt) node[right] {$2$};
\fill (0.30901699437494745,0.9510565162951535,0.0) circle (0.3pt) node[right] {$1$};
\fill (-0.8090169943749473,0.5877852522924731,0.0) circle (0.3pt) node[above] {$2$};
\fill (-0.8090169943749473,-0.5877852522924731,0.0) circle (0.3pt) node[left] {$3$};
\fill (0.3090169943749473,-0.9510565162951535,0.0) circle (0.3pt) node[below] {$1$};

\fill (0.4363389981249825,0.31701883876505116,0.4) circle (0.3pt) node[above] {$3$};
\fill (-0.16666666666666663,0.5129472561958756,0.4) circle (0.3pt) node[above] {$3$};
\fill (-0.6393446629166316,3.700743415417188e-17,0.4) circle (0.3pt) node[above] {$1$};

\fill (0.4363389981249824,-0.31701883876505116,0.4) circle (0.3pt) node[above] {$3$};
\fill (-0.16666666666666674,-0.5129472561958756,0.4) circle (0.3pt) node[above] {};
\fill (-0.16666666666666674,-0.5129472561958756,0.36) circle (0pt) node[above] {$2$};

\draw[dotted] (0,0,0) -- (1.0,0.0,0.0);
\draw[dotted] (0,0,0) -- (0.30901699437494745,0.9510565162951535,0.0);
\draw (0,0,0) -- (0.4363389981249825,0.31701883876505116,0.4);
\draw (0,0,0) -- (-0.8090169943749473,0.5877852522924731,0.0);
\draw (0,0,0) -- (-0.16666666666666663,0.5129472561958756,0.4);
\draw[dotted] (0,0,0) -- (-0.8090169943749475,-0.587785252292473,0.0);
\draw (0,0,0) -- (-0.6393446629166316,3.700743415417188e-17,0.4);
\draw (0,0,0) -- (0.3090169943749473,-0.9510565162951535,0.0);
\draw (0,0,0) -- (-0.16666666666666674,-0.5129472561958756,0.4);
\draw (0,0,0) -- (0.4363389981249824,-0.31701883876505116,0.4);
\draw (1.0,0.0,0.0) -- (0.30901699437494745,0.9510565162951535,0.0);
\draw (1.0,0.0,0.0) -- (0.4363389981249825,0.31701883876505116,0.4);
\draw (1.0,0.0,0.0) -- (0.3090169943749473,-0.9510565162951535,0.0);
\draw (1.0,0.0,0.0) -- (0.4363389981249824,-0.31701883876505116,0.4);
\draw (0.30901699437494745,0.9510565162951535,0.0) -- (0.4363389981249825,0.31701883876505116,0.4);
\draw[dotted] (0.30901699437494745,0.9510565162951535,0.0) -- (-0.8090169943749473,0.5877852522924731,0.0);
\draw (0.30901699437494745,0.9510565162951535,0.0) -- (-0.16666666666666663,0.5129472561958756,0.4);
\draw (-0.8090169943749473,0.5877852522924731,0.0) -- (-0.16666666666666663,0.5129472561958756,0.4);
\draw[dotted] (-0.8090169943749473,0.5877852522924731,0.0) -- (-0.8090169943749475,-0.587785252292473,0.0);
\draw (-0.8090169943749473,0.5877852522924731,0.0) -- (-0.6393446629166316,3.700743415417188e-17,0.4);
\draw (-0.8090169943749475,-0.587785252292473,0.0) -- (-0.6393446629166316,3.700743415417188e-17,0.4);
\draw (-0.8090169943749475,-0.587785252292473,0.0) -- (0.3090169943749473,-0.9510565162951535,0.0);
\draw (-0.8090169943749475,-0.587785252292473,0.0) -- (-0.16666666666666674,-0.5129472561958756,0.4);
\draw (0.3090169943749473,-0.9510565162951535,0.0) -- (-0.16666666666666674,-0.5129472561958756,0.4);
\draw (0.3090169943749473,-0.9510565162951535,0.0) -- (0.4363389981249824,-0.31701883876505116,0.4);
\end{tikzpicture}

  \end{minipage}
    \vspace*{-1cm}
    
  \caption{$4$-coloring of the non-colorable triangulation of Figure~\ref{fig:noncolorable} and the coloring of the virtual extension to~$\setR^3$, see Remark~\ref{rem:virtual-extension}.}
  \label{fig:noncolorable-lifted}
\end{figure}
\begin{remark}[No equivalence with virtual extension]
  \label{rem:problem-virtual}
  Although our bisection algorithm for generalized colored initial triangulations~$\tria_0$ is inspired by the virtual extension explained after Definition~\ref{def:gcoloring}, one has to be careful with this analogy: Let us start with an $N{+}1$-colored initial triangulation~$(\tria_0,\frc)$ and let $\tria_0^+$ denote the colored, virtual extension of~$\tria_0$ to~$\RRN$. Then the restriction of~$\Bisec(\tria_0^+)$ to~$\tria_0$ does for $N> n$ in general not agree with~$\BisecT$.
Figure~\ref{fig:problem-virtual} illustrates this phenomenon for~$n=2$ and $N=3$, where~$\tria_0$ just consists of the triangle~$T$ spanned by the vertices~$0$, $1$, and $2$ (we use the colors as labels for the vertices). The complete tetrahedron~$T^+$ with all four vertices $0$, $1$, $2$, and $3$ is the virtual extension~$\tria_0^+ = \lbrace T^+\rbrace$. Note that $T^+$ is just the Kuhn simplex spanned by~$0,e_1, e_1+e_2,e_1+e_2+e_3$ with standard coloring $3,0,1,2$. The triangle~$T$ is the face of~$T^+$ in the hyperplane $\set{x_1=1}$. The pictures are rotated for better visibility such that the hyperplane~$\set{x_1=1}$ agrees with the surface of the drawing plane.
We repeatedly refine the 3-simplex in~$\tria_0^+$ or the triangle in~$\tria$  respectively at the point $P\coloneqq (1, {7}/{16},{3}/{16})$. The left three pictures show three consecutive refinements of~$\tria_0^+$. Note that in the second picture the triangles within $\set{x_1=1}$ that contain the point~$P$ are not bisected, but some interior tetrahedra are bisected and its conformal closure bisects the simplex at the vertex~$2$.  However, it can be seen in the last picture that the algorithm for the colored~$\tria_0$ does not need to bisect this triangle at the vertex~$2$. Hence, the triangulation in the last picture from~$\BisecT$ is not the restriction of any triangulation from~$\Bisec(\tria_0^+)$ to the hyperplane~$\set{x_1=1}$.
  \begin{figure}
    \begin{minipage}{.25\textwidth}%
      \begin{center}%
        \tdplotsetmaincoords{85}{80}
\begin{tikzpicture}[scale = 2,tdplot_main_coords]
\fill (0,0,0) circle (0.3pt) node[above] {$3$};
\fill (1,0,0) circle (0.3pt) node[below] {$0$};
\fill (1,1,0) circle (0.3pt) node[below] {$1$};
\fill (1,1,1) circle (0.3pt) node[above] {$2$};
\draw [draw=black] (0.5,0.0,0.0) -- (0.5,0.5,0.5) -- (0.0,0.0,0.0) -- cycle;
\draw [draw=black] (0.75,0.25,0.25) -- (0.5,0.0,0.0) -- (1.0,0.0,0.0) -- cycle;
\draw [draw=black] (1.0,0.25,0.25) -- (0.75,0.25,0.25) -- (1.0,0.0,0.0) -- cycle;
\draw [draw=black] (0.75,0.75,0.75) -- (1.0,0.5,0.5) -- (1.0,1.0,1.0) -- cycle;
\draw [draw=black] (0.75,0.25,0.25) -- (0.5,0.0,0.0) -- (0.5,0.5,0.5) -- cycle;
\draw [draw=black] (0.75,0.5,0.5) -- (0.75,0.25,0.25) -- (0.5,0.5,0.5) -- cycle;
\draw [draw=black] (0.75,0.5,0.5) -- (0.75,0.75,0.75) -- (0.5,0.5,0.5) -- cycle;
\draw [draw=black] (0.875,0.375,0.375) -- (0.75,0.5,0.5) -- (1.0,0.5,0.5) -- cycle;
\draw [draw=black] (1.0,0.375,0.375) -- (0.875,0.375,0.375) -- (1.0,0.5,0.5) -- cycle;
\draw [draw=black] (0.75,0.5,0.5) -- (0.75,0.75,0.75) -- (1.0,0.5,0.5) -- cycle;
\draw [draw=black] (0.875,0.375,0.375) -- (1.0,0.25,0.25) -- (0.75,0.25,0.25) -- cycle;
\draw [draw=black] (0.875,0.375,0.375) -- (0.75,0.5,0.5) -- (0.75,0.25,0.25) -- cycle;
\draw [draw=black] (1.0,0.375,0.375) -- (0.875,0.375,0.375) -- (1.0,0.25,0.25) -- cycle;
\draw [draw=black, fill=white] (1.0,0.25,0.0) -- (1.0,0.25,0.25) -- (1.0,0.0,0.0) -- cycle;
\draw [draw=black, fill=white] (1.0,0.75,0.25) -- (1.0,1.0,0.5) -- (1.0,1.0,0.0) -- cycle;
\draw [draw=black, fill=white] (1.0,0.75,0.0) -- (1.0,0.75,0.25) -- (1.0,1.0,0.0) -- cycle;
\draw [draw=black, fill=white] (1.0,1.0,0.5) -- (1.0,0.5,0.5) -- (1.0,1.0,1.0) -- cycle;
\draw [draw=black, fill=white] (1.0,0.5,0.25) -- (1.0,0.75,0.25) -- (1.0,0.5,0.5) -- cycle;
\draw [draw=black, fill=white] (1.0,0.375,0.375) -- (1.0,0.5,0.25) -- (1.0,0.5,0.5) -- cycle;
\draw [draw=black, fill=white] (1.0,0.75,0.25) -- (1.0,1.0,0.5) -- (1.0,0.5,0.5) -- cycle;
\draw [draw=black, fill=white] (1.0,0.375,0.125) -- (1.0,0.25,0.0) -- (1.0,0.5,0.0) -- cycle;
\draw [draw=black, fill=white] (1.0,0.5,0.125) -- (1.0,0.375,0.125) -- (1.0,0.5,0.0) -- cycle;
\draw [draw=black, fill=white] (1.0,0.5,0.125) -- (1.0,0.625,0.125) -- (1.0,0.5,0.0) -- cycle;
\draw [draw=black, fill=white] (1.0,0.625,0.125) -- (1.0,0.75,0.0) -- (1.0,0.5,0.0) -- cycle;
\draw [draw=black, fill=white] (1.0,0.375,0.125) -- (1.0,0.25,0.0) -- (1.0,0.25,0.25) -- cycle;
\draw [draw=black, fill=white] (1.0,0.375,0.25) -- (1.0,0.375,0.125) -- (1.0,0.25,0.25) -- cycle;
\draw [draw=black, fill=white] (1.0,0.375,0.25) -- (1.0,0.375,0.375) -- (1.0,0.25,0.25) -- cycle;
\draw [draw=black, fill=white] (1.0,0.625,0.125) -- (1.0,0.5,0.25) -- (1.0,0.75,0.25) -- cycle;
\draw [draw=black, fill=white] (1.0,0.5,0.125) -- (1.0,0.375,0.125) -- (1.0,0.5,0.25) -- cycle;
\draw [draw=black, fill=white] (1.0,0.375,0.25) -- (1.0,0.375,0.125) -- (1.0,0.5,0.25) -- cycle;
\draw [draw=black, fill=white] (1.0,0.5,0.125) -- (1.0,0.625,0.125) -- (1.0,0.5,0.25) -- cycle;
\draw [draw=black, fill=white] (1.0,0.375,0.25) -- (1.0,0.375,0.375) -- (1.0,0.5,0.25) -- cycle;
\draw [draw=black, fill=white] (1.0,0.625,0.125) -- (1.0,0.75,0.0) -- (1.0,0.75,0.25) -- cycle;
\fill (1,7/16,3/16) circle (.6pt);
\end{tikzpicture}%
      \end{center}%
    \end{minipage}%
    \begin{minipage}{.25\textwidth}%
      \begin{center}%
        \tdplotsetmaincoords{85}{80}
\begin{tikzpicture}[scale = 2,tdplot_main_coords]
\fill (0,0,0) circle (0.3pt) node[above] {$3$};
\fill (1,0,0) circle (0.3pt) node[below] {$0$};
\fill (1,1,0) circle (0.3pt) node[below] {$1$};
\fill (1,1,1) circle (0.3pt) node[above] {$2$};
\draw [draw=black] (0.5,0.0,0.0) -- (0.5,0.5,0.5) -- (0.0,0.0,0.0) -- cycle;
\draw [draw=black] (0.75,0.25,0.25) -- (0.5,0.0,0.0) -- (1.0,0.0,0.0) -- cycle;
\draw [draw=black] (1.0,0.25,0.25) -- (0.75,0.25,0.25) -- (1.0,0.0,0.0) -- cycle;
\draw [draw=black] (1.0,0.75,0.75) -- (0.75,0.75,0.75) -- (1.0,1.0,1.0) -- cycle;
\draw [draw=black] (0.75,0.25,0.25) -- (0.5,0.0,0.0) -- (0.5,0.5,0.5) -- cycle;
\draw [draw=black] (0.75,0.5,0.5) -- (0.75,0.25,0.25) -- (0.5,0.5,0.5) -- cycle;
\draw [draw=black] (0.75,0.5,0.5) -- (0.75,0.75,0.75) -- (0.5,0.5,0.5) -- cycle;
\draw [draw=black] (0.875,0.375,0.375) -- (0.75,0.5,0.5) -- (1.0,0.5,0.5) -- cycle;
\draw [draw=black] (1.0,0.375,0.375) -- (0.875,0.375,0.375) -- (1.0,0.5,0.5) -- cycle;
\draw [draw=black] (0.75,0.5,0.5) -- (0.75,0.75,0.75) -- (1.0,0.5,0.5) -- cycle;
\draw [draw=black] (1.0,0.75,0.75) -- (0.75,0.75,0.75) -- (1.0,0.5,0.5) -- cycle;
\draw [draw=black] (0.875,0.375,0.375) -- (1.0,0.25,0.25) -- (0.75,0.25,0.25) -- cycle;
\draw [draw=black] (0.875,0.375,0.375) -- (0.75,0.5,0.5) -- (0.75,0.25,0.25) -- cycle;
\draw [draw=black] (1.0,0.375,0.375) -- (0.875,0.375,0.375) -- (1.0,0.25,0.25) -- cycle;
\draw [draw=black, fill=white] (1.0,0.25,0.0) -- (1.0,0.25,0.25) -- (1.0,0.0,0.0) -- cycle;
\draw [draw=black, fill=white] (1.0,0.75,0.25) -- (1.0,1.0,0.5) -- (1.0,1.0,0.0) -- cycle;
\draw [draw=black, fill=white] (1.0,0.75,0.0) -- (1.0,0.75,0.25) -- (1.0,1.0,0.0) -- cycle;
\draw [draw=black, fill=white] (1.0,0.75,0.75) -- (1.0,1.0,0.5) -- (1.0,1.0,1.0) -- cycle;
\draw [draw=black, fill=white] (1.0,0.5,0.25) -- (1.0,0.75,0.25) -- (1.0,0.5,0.5) -- cycle;
\draw [draw=black, fill=white] (1.0,0.375,0.375) -- (1.0,0.5,0.25) -- (1.0,0.5,0.5) -- cycle;
\draw [draw=black, fill=white] (1.0,0.75,0.5) -- (1.0,0.75,0.25) -- (1.0,0.5,0.5) -- cycle;
\draw [draw=black, fill=white] (1.0,0.75,0.5) -- (1.0,0.75,0.75) -- (1.0,0.5,0.5) -- cycle;
\draw [draw=black, fill=white] (1.0,0.375,0.125) -- (1.0,0.25,0.0) -- (1.0,0.5,0.0) -- cycle;
\draw [draw=black, fill=white] (1.0,0.5,0.125) -- (1.0,0.375,0.125) -- (1.0,0.5,0.0) -- cycle;
\draw [draw=black, fill=white] (1.0,0.5,0.125) -- (1.0,0.625,0.125) -- (1.0,0.5,0.0) -- cycle;
\draw [draw=black, fill=white] (1.0,0.625,0.125) -- (1.0,0.75,0.0) -- (1.0,0.5,0.0) -- cycle;
\draw [draw=black, fill=white] (1.0,0.375,0.125) -- (1.0,0.25,0.0) -- (1.0,0.25,0.25) -- cycle;
\draw [draw=black, fill=white] (1.0,0.375,0.25) -- (1.0,0.375,0.125) -- (1.0,0.25,0.25) -- cycle;
\draw [draw=black, fill=white] (1.0,0.375,0.25) -- (1.0,0.375,0.375) -- (1.0,0.25,0.25) -- cycle;
\draw [draw=black, fill=white] (1.0,0.625,0.125) -- (1.0,0.5,0.25) -- (1.0,0.75,0.25) -- cycle;
\draw [draw=black, fill=white] (1.0,0.5,0.125) -- (1.0,0.375,0.125) -- (1.0,0.5,0.25) -- cycle;
\draw [draw=black, fill=white] (1.0,0.375,0.25) -- (1.0,0.375,0.125) -- (1.0,0.5,0.25) -- cycle;
\draw [draw=black, fill=white] (1.0,0.5,0.125) -- (1.0,0.625,0.125) -- (1.0,0.5,0.25) -- cycle;
\draw [draw=black, fill=white] (1.0,0.375,0.25) -- (1.0,0.375,0.375) -- (1.0,0.5,0.25) -- cycle;
\draw [draw=black, fill=white] (1.0,0.75,0.5) -- (1.0,0.75,0.25) -- (1.0,1.0,0.5) -- cycle;
\draw [draw=black, fill=white] (1.0,0.75,0.5) -- (1.0,0.75,0.75) -- (1.0,1.0,0.5) -- cycle;
\draw [draw=black, fill=white] (1.0,0.625,0.125) -- (1.0,0.75,0.0) -- (1.0,0.75,0.25) -- cycle;
\fill (1,7/16,3/16) circle (.6pt);
\end{tikzpicture}%
      \end{center}%
    \end{minipage}%
    \begin{minipage}{.25\textwidth}%
      \begin{center}%
        \tdplotsetmaincoords{85}{80}
\begin{tikzpicture}[scale = 2,tdplot_main_coords]
\fill (0,0,0) circle (0.3pt) node[above] {$3$};
\fill (1,0,0) circle (0.3pt) node[below] {$0$};
\fill (1,1,0) circle (0.3pt) node[below] {$1$};
\fill (1,1,1) circle (0.3pt) node[above] {$2$};
\draw [draw=black] (0.5,0.0,0.0) -- (0.5,0.5,0.5) -- (0.0,0.0,0.0) -- cycle;
\draw [draw=black] (0.75,0.25,0.25) -- (0.5,0.0,0.0) -- (1.0,0.0,0.0) -- cycle;
\draw [draw=black] (1.0,0.25,0.25) -- (0.75,0.25,0.25) -- (1.0,0.0,0.0) -- cycle;
\draw [draw=black] (1.0,0.75,0.75) -- (0.75,0.75,0.75) -- (1.0,1.0,1.0) -- cycle;
\draw [draw=black] (0.75,0.25,0.25) -- (0.5,0.0,0.0) -- (0.5,0.5,0.5) -- cycle;
\draw [draw=black] (0.75,0.5,0.5) -- (0.75,0.25,0.25) -- (0.5,0.5,0.5) -- cycle;
\draw [draw=black] (0.75,0.5,0.5) -- (0.75,0.75,0.75) -- (0.5,0.5,0.5) -- cycle;
\draw [draw=black] (0.875,0.375,0.375) -- (0.75,0.5,0.5) -- (1.0,0.5,0.5) -- cycle;
\draw [draw=black] (1.0,0.375,0.375) -- (0.875,0.375,0.375) -- (1.0,0.5,0.5) -- cycle;
\draw [draw=black] (0.75,0.5,0.5) -- (0.75,0.75,0.75) -- (1.0,0.5,0.5) -- cycle;
\draw [draw=black] (1.0,0.75,0.75) -- (0.75,0.75,0.75) -- (1.0,0.5,0.5) -- cycle;
\draw [draw=black] (0.875,0.375,0.375) -- (1.0,0.25,0.25) -- (0.75,0.25,0.25) -- cycle;
\draw [draw=black] (0.875,0.375,0.375) -- (0.75,0.5,0.5) -- (0.75,0.25,0.25) -- cycle;
\draw [draw=black] (1.0,0.375,0.375) -- (0.875,0.375,0.375) -- (1.0,0.25,0.25) -- cycle;
\draw [draw=black, fill=white] (1.0,0.25,0.0) -- (1.0,0.25,0.25) -- (1.0,0.0,0.0) -- cycle;
\draw [draw=black, fill=white] (1.0,0.75,0.25) -- (1.0,1.0,0.5) -- (1.0,1.0,0.0) -- cycle;
\draw [draw=black, fill=white] (1.0,0.75,0.0) -- (1.0,0.75,0.25) -- (1.0,1.0,0.0) -- cycle;
\draw [draw=black, fill=white] (1.0,0.75,0.75) -- (1.0,1.0,0.5) -- (1.0,1.0,1.0) -- cycle;
\draw [draw=black, fill=white] (1.0,0.5,0.25) -- (1.0,0.75,0.25) -- (1.0,0.5,0.5) -- cycle;
\draw [draw=black, fill=white] (1.0,0.375,0.375) -- (1.0,0.5,0.25) -- (1.0,0.5,0.5) -- cycle;
\draw [draw=black, fill=white] (1.0,0.75,0.5) -- (1.0,0.75,0.25) -- (1.0,0.5,0.5) -- cycle;
\draw [draw=black, fill=white] (1.0,0.75,0.5) -- (1.0,0.75,0.75) -- (1.0,0.5,0.5) -- cycle;
\draw [draw=black, fill=white] (1.0,0.375,0.125) -- (1.0,0.25,0.0) -- (1.0,0.5,0.0) -- cycle;
\draw [draw=black, fill=white] (1.0,0.5,0.125) -- (1.0,0.375,0.125) -- (1.0,0.5,0.0) -- cycle;
\draw [draw=black, fill=white] (1.0,0.5,0.125) -- (1.0,0.625,0.125) -- (1.0,0.5,0.0) -- cycle;
\draw [draw=black, fill=white] (1.0,0.625,0.125) -- (1.0,0.75,0.0) -- (1.0,0.5,0.0) -- cycle;
\draw [draw=black, fill=white] (1.0,0.375,0.125) -- (1.0,0.25,0.0) -- (1.0,0.25,0.25) -- cycle;
\draw [draw=black, fill=white] (1.0,0.375,0.25) -- (1.0,0.375,0.125) -- (1.0,0.25,0.25) -- cycle;
\draw [draw=black, fill=white] (1.0,0.375,0.25) -- (1.0,0.375,0.375) -- (1.0,0.25,0.25) -- cycle;
\draw [draw=black, fill=white] (1.0,0.625,0.125) -- (1.0,0.5,0.25) -- (1.0,0.75,0.25) -- cycle;
\draw [draw=black, fill=white] (1.0,0.5,0.125) -- (1.0,0.625,0.125) -- (1.0,0.5,0.25) -- cycle;
\draw [draw=black, fill=white] (1.0,0.4375,0.1875) -- (1.0,0.5,0.125) -- (1.0,0.5,0.25) -- cycle;
\draw [draw=black, fill=white] (1.0,0.375,0.25) -- (1.0,0.375,0.375) -- (1.0,0.5,0.25) -- cycle;
\draw [draw=black, fill=white] (1.0,0.4375,0.1875) -- (1.0,0.375,0.25) -- (1.0,0.5,0.25) -- cycle;
\draw [draw=black, fill=white] (1.0,0.75,0.5) -- (1.0,0.75,0.25) -- (1.0,1.0,0.5) -- cycle;
\draw [draw=black, fill=white] (1.0,0.75,0.5) -- (1.0,0.75,0.75) -- (1.0,1.0,0.5) -- cycle;
\draw [draw=black, fill=white] (1.0,0.625,0.125) -- (1.0,0.75,0.0) -- (1.0,0.75,0.25) -- cycle;
\draw [draw=black, fill=white] (1.0,0.4375,0.1875) -- (1.0,0.5,0.125) -- (1.0,0.375,0.125) -- cycle;
\draw [draw=black, fill=white] (1.0,0.4375,0.1875) -- (1.0,0.375,0.25) -- (1.0,0.375,0.125) -- cycle;
\fill (1,7/16,3/16) circle (.6pt);
\end{tikzpicture}%
      \end{center}%
    \end{minipage}%
    \begin{minipage}{.25\textwidth}%
      \begin{center}%
        \tdplotsetmaincoords{85}{80}
\begin{tikzpicture}[scale = 2,tdplot_main_coords]
\fill (1,0,0) circle (0.3pt) node[below] {$0$};
\fill (1,1,0) circle (0.3pt) node[below] {$1$};
\fill (1,1,1) circle (0.3pt) node[above] {$2$};
\node (A) at  (0,0,0) {$ $};
\draw [draw=black, fill=white] (1.0,0.25,0.0) -- (1.0,0.25,0.25) -- (1.0,0.0,0.0) -- cycle;
\draw [draw=black, fill=white] (1.0,0.75,0.25) -- (1.0,1.0,0.5) -- (1.0,1.0,0.0) -- cycle;
\draw [draw=black, fill=white] (1.0,0.75,0.0) -- (1.0,0.75,0.25) -- (1.0,1.0,0.0) -- cycle;
\draw [draw=black, fill=white] (1.0,1.0,0.5) -- (1.0,0.5,0.5) -- (1.0,1.0,1.0) -- cycle;
\draw [draw=black, fill=white] (1.0,0.5,0.25) -- (1.0,0.75,0.25) -- (1.0,0.5,0.5) -- cycle;
\draw [draw=black, fill=white] (1.0,0.375,0.375) -- (1.0,0.5,0.25) -- (1.0,0.5,0.5) -- cycle;
\draw [draw=black, fill=white] (1.0,0.75,0.25) -- (1.0,1.0,0.5) -- (1.0,0.5,0.5) -- cycle;
\draw [draw=black, fill=white] (1.0,0.375,0.125) -- (1.0,0.25,0.0) -- (1.0,0.5,0.0) -- cycle;
\draw [draw=black, fill=white] (1.0,0.5,0.125) -- (1.0,0.375,0.125) -- (1.0,0.5,0.0) -- cycle;
\draw [draw=black, fill=white] (1.0,0.5,0.125) -- (1.0,0.625,0.125) -- (1.0,0.5,0.0) -- cycle;
\draw [draw=black, fill=white] (1.0,0.625,0.125) -- (1.0,0.75,0.0) -- (1.0,0.5,0.0) -- cycle;
\draw [draw=black, fill=white] (1.0,0.375,0.125) -- (1.0,0.25,0.0) -- (1.0,0.25,0.25) -- cycle;
\draw [draw=black, fill=white] (1.0,0.375,0.25) -- (1.0,0.375,0.125) -- (1.0,0.25,0.25) -- cycle;
\draw [draw=black, fill=white] (1.0,0.375,0.25) -- (1.0,0.375,0.375) -- (1.0,0.25,0.25) -- cycle;
\draw [draw=black, fill=white] (1.0,0.625,0.125) -- (1.0,0.5,0.25) -- (1.0,0.75,0.25) -- cycle;
\draw [draw=black, fill=white] (1.0,0.5,0.125) -- (1.0,0.375,0.125) -- (1.0,0.5,0.25) -- cycle;
\draw [draw=black, fill=white] (1.0,0.375,0.25) -- (1.0,0.375,0.125) -- (1.0,0.5,0.25) -- cycle;
\draw [draw=black, fill=white] (1.0,0.5,0.125) -- (1.0,0.625,0.125) -- (1.0,0.5,0.25) -- cycle;
\draw [draw=black, fill=white] (1.0,0.375,0.25) -- (1.0,0.375,0.375) -- (1.0,0.5,0.25) -- cycle;
\draw [draw=black, fill=white] (1.0,0.625,0.125) -- (1.0,0.75,0.0) -- (1.0,0.75,0.25) -- cycle;
\draw [draw=black] (1.0,0.5,0.125) -- (1.0,0.375,0.25);
\fill (1,7/16,3/16) circle (.6pt);
\end{tikzpicture}%
      \end{center}%
    \end{minipage}
      \vspace*{-1cm}
    
    \caption{Refinements of of~$\tria_0$ (last picture) and refinements of the virtual extension~$\tria_0^+$ (first three pictures) induced by the displayed refinement point $(1,7/16,3/16)$ as explained in Remark~\ref{rem:problem-virtual}. The picture is rotated such that the plane~$\set{x_1=1}$ is in the front.}
    \label{fig:problem-virtual}
  \end{figure}
\end{remark}
\begin{remark}[Relation to \cite{AGK.2018}]\label{rem:relationToAGK18}
Alkämper, Gaspoz, and Kl\"ofkorn suggest an alternative initialization that splits the vertices $\vertices(\tria_0)$ of the initial triangulation into two disjoint sets $\vertices_0$ and $\vertices_1$ and provides global orderings for each of them in \cite{AGK.2018}. In our terms $\vertices_0$ contains the vertices of levels $-1$ and $0$ and $\vertices_1$ contains the vertices of level $1$ respectively. The order of the vertices in the tagged simplex is obtained by restriction of these global orderings to the simplex vertices and concatenation of the level 0 and the level 1 vertices.
If one chooses $\vertices_0 \coloneqq \vertices(\tria_0)$, their algorithm does the same as ours if we color the vertices with as many colors as vertices in $\vertices(\tria_0)$.
\end{remark}
\begin{remark}[4-coloring in 2D]\label{rem:4Coloring}
If $n=2$, the four color theorem states that we can always obtain a coloring with $N + 1 = 4$ colors. 
However, the 2-dimensional case can be solved with different approaches as discussed in Section~\ref{subsec:color-init-triang}. 
In higher dimensions, there exists no natural analog to the four color theorem, that is, in general the number of colors is unbounded \cite[Chapter IV]{Tietze65}.
However, due to the additional structure of the mesh we obtain with Algorithm~\ref{algo:coloring} a coloring where the number of colors can be bounded in terms of the shape regularity, see Lemma~\ref{lem:largestColor}.
\end{remark}
\subsection{Verification of Theorems~\ref{thm:basic-properties} and \ref{thm:closuregcolored}}\label{subsec:ClosureEstimate}
Even though the lifted triangulation is not needed in our computations, we exploit the notion of generation $\gen$ for triangulations in $\mathbb{R}^N$. In particular, let $(\tria_0,\frc)$ be an $N{+}1$-colored triangulation. We set for each initial vertex $v \in \vertices(\tria_0)$ its generation and level by
\begin{equation*}
\genN(v) \coloneqq -\frc(v)\qquad\text{and}\qquad \levelN(v) \coloneqq \begin{cases}
0&\text{if }\frc(v) < N,\\
-1&\text{if }\frc(v) = N.
\end{cases}
\end{equation*}
Algorithm~\ref{algo:subsimplex} assigns to each new vertex $v$ a generation~$\genN$.
This generation is unique since it only depends on the generations of the vertices in the bisection edge.
Level~$\levelN(v) \in \setZ$ and type~$\typeN(v) \in \set{1,\dots, N}$ are defined according to \eqref{eq:leveltype}.
We define $\BisecT = \bisec(\tria_0,\frc)$ as the set of all triangulations obtained by successive application of Algorithm~\ref{algo:closure-recursive} with the bisection rule of Algorithm~\ref{algo:subsimplex} (or equivalently Algorithm~\ref{algo:maubach} as stated in Theorem~\ref{thm:EquiRefinements}) to an initial triangulation~$\tria_0$ with $N{+}1$-coloring $\frc$.
We explain in the proof of Theorem~\ref{thm:basic-properties} below why the algorithm terminates.
We denote the set of all possible $n$-simplices resulting from successive bisections of initial simplices $T_0\in \tria_0$ by~$\mtree = \bigcup \BisecT$. 
The equivalence \cite[Sec.~3.3]{DieningStornTscherpel23} of the well-posed bisection routine in $\tria_0^+$ and the refinement of its subsimplices in Algorithm~\ref{algo:subsimplex} verifies the following lemma.
\begin{lemma}[Well-posedness]\label{lem:well-posed}
The generation, level, and type of each bisection vertex $b = \bsv(T)$ defined in Algorithm~\ref{algo:subsimplex} are unique, that is, they are the same for all applications of Algorithm~\ref{algo:subsimplex} to simplices $T \in \mtree$ with bisection vertex $b$.
Moreover, the generation of $\bsv(T)$ is strictly larger than the generation of $T$ in the sense that
\begin{align*}
\max_{v\in \vertices(T)} \genN(v) < \genN(\bsv(T)). 
\end{align*}
\end{lemma}
We write $\tria_2 \leq \tria_1$ if $\tria_1$ is a refinement of~$\tria_2$ for $\tria_1,\tria_2\in \BisecT$. 
This makes~$\BisecT$ a partially ordered set. 
\begin{lemma}[Refinement chains]\label{lem:RefChainsNew}
Let $T \in \tria \in \BisecT$ be a simplex.
\begin{enumerate}
\item Suppose that $T' \in \mtree$ is flagged for refinement in a recursive call of Algorithm~\ref{algo:closure-recursive}. Then there exists a chain of simplices $T_0,\dots,T_J \in \mtree$ with $J\in \mathbb{N}_0$, $T_0 = T$, and $T_J = T'$ such that the bisection edges satisfy \label{itm:RefChains1} 
\begin{align}\label{eq:ChainProp}
\bse(T_j) \in  \edges(T_{j+1})\qquad\text{for all }j= 0,\dots,J{-}1.
\end{align}
\item If $T'\in \mtree$ is a simplex that resulted from Algorithm~\ref{algo:closure-recursive}, then there exists a chain of simplices $T_0,\dots,T_J\in\mtree$ with $J\in \mathbb{N}_0$, $T_0 = T$, and \eqref{eq:ChainProp} such that $T'$ is a child of $T_J$.  \label{itm:RefChains2}
\end{enumerate}
\end{lemma}
\begin{proof}
Let $T_0 \coloneqq T \in \BisecT$ and suppose that $T' \in \mtree$ is flagged for refinement in a recursive call of Algorithm~\ref{algo:closure-recursive}.  By definition Algorithm~\ref{algo:closure-recursive} leads to a sequence of flagged simplices $T_0,\dots,T_{L_1} \in \tria$ where $T_{L_1}$ with $L_1\in \mathbb{N}_0\cup \lbrace \infty \rbrace$ denotes the first simplex that causes bisections in the sense that there is no $T'' \in \omega_\tria(\bse(T_{L_1}))$ with $\bse(T'') \neq \bse(T)$.
The chain satisfies
\begin{align*}
\bse(T_j) \in \edges(T_{j+1})\qquad\text{for all }j=0,\dots,L_1-1.
\end{align*}
If $T'$ is within the chain $T_0,\dots,T_{L_1}$, this shows the statement. Otherwise, the algorithm bisects all simplices in the edge patch $\omega_\tria(\bse(T_{L_1}))$. This leads to a new regular triangulation $\tria'$ that contains the simplices $T_0,\dots,T_{L_1-1}$. The inductive routine in Algorithm~\ref{algo:closure-recursive} proceeds to flag simplices $T'_{L_1},\dots,T'_{L_1 - 1 + L_2} \in \tria'$ for refinement with $L_2 \in \mathbb{N}_0 \cup \lbrace \infty \rbrace$ such that $\bse(T_{L_1 - 1}) \in \edges(T'_{L_1})$ and 
\begin{align*}
\bse(T'_{L_1 + \ell}) \in \edges(T'_{L_1 + \ell + 1})\qquad\text{for all } \ell = 0,\dots,L_2-2.
\end{align*}
This leads to a new sequence of simplices $T_0,\dots, T_{L_1 - 1}, T'_{L_1},\dots,T'_{L_1 - 1+L_2} \in \tria' \in \BisecT$. If $T'$ is within this chain, this proves the claim in \ref{itm:RefChains1}. Otherwise, we proceed inductively.
The statement in \ref{itm:RefChains2} follows by \ref{itm:RefChains1}.
\end{proof} 
%
%
  As in~\eqref{eq:GenMaxVerticesSub} we set generation and level of each $m$-subsimplex $S = \simplex{v_0,\dots, v_m}$ as
  \begin{align}\label{eq:gN-for-sub}
    \genN(S) &\coloneqq \gen(v_0) =  \max_{v \in \vertices(S)} \genN(v)\quad\text{and}\quad   \levelN(S) \coloneqq \max_{v \in \vertices(S)} \levelN(v).
  \end{align}
Our notion of generation $\genN(T)$ for $T\in \mtree$ does not coincide with the notion of generation $\mathtt{gen}(T)$ defined in many publications as number of bisections of an initial simplex $T_0\in \tria_0$ needed to create $T$.   
However, the level coincides with the notion of level used for example in \cite{DieningStornTscherpel23}, that is, 
  \begin{align*}
    \levelN(T) = \lceil \gen(T)/N \rceil  &= \lceil \mathtt{gen}(T)/n \rceil \eqqcolon \mathtt{lvl}(T) \qquad \text{for all $T \in \mtree$.}
  \end{align*}
  The identity follows from the fact that for given~$T_0 \in \tria_0$ both~$\levelN$ and $\level$ increase at the first bisection and then exactly after every $n$-th consecutive bisection. Additionally this shows that the $\Refine(\tria,T)$ routine with $T\in \tria \in \BisecT$ increases the level of the descendants of $T$ by at most one, that is,
  \begin{align}\label{eq:levelIncrease}
  \level(T') + 1 \leq \level(T)\qquad\text{for all }T' \in \Refine(\tria,T)\text{ with }T'\subset T.
\end{align}   
Without proof we state the following simple facts for the level 
\begin{equation*}
\level(S) \coloneqq \levelN(S) \coloneqq  \max_{v \in \vertices(S)} \levelN(v).
\end{equation*}
\begin{lemma}[Levels in subsimplices]
  \label{lem:too-simple-to-be-proved}
  Let $\tria_0$ be $N{+}1$-colored.
  \begin{enumerate}
  \item \label{lem:too-simple-to-be-proved-1}%
    If $T \in \mtree$ and $S= \simplex{v_0,\dots, v_m}\subset T$ is a subsimplex, then
    \begin{align*}
      \level(v_0) \geq \level(v_1) \geq \dots \geq \level(v_{m-1}) \geq \level(v_0)-1.
    \end{align*}
  \item \label{lem:too-simple-to-be-proved-2}%
    If $T \in \mtree$ and $S= \simplex{v_0,\dots, v_m}\subset T$ is a subsimplex, then
    \begin{align*}
      \level(\bsv(S)) =  \level(\bse(S)) +1 = \level(v_{m-1}) +1.
    \end{align*}
  \end{enumerate}
\end{lemma}
In order to show that the bisection routine defined by Algorithms~\ref{algo:subsimplex} and~\ref{algo:closure-recursive} terminates, we need a few auxiliary results. We set for any $m$-(sub)simplex~$S$ with~$m \geq 1$
\begin{align*}
  \genNsharp(S) \coloneqq \genN(\bsv(S))\qquad\text{and} \qquad
  \levelNsharp(S) \coloneqq \level(\bsv(S)).
\end{align*}
The following lemma states that the bisection edge of a simplex, which is refined before the other edges of the simplex, is the oldest one with respect to~$\genNsharp$. Note that for~$n \geq 3$ the bisection edge is in general not the oldest one with respect to~$\genN$.
\begin{lemma}[Unique $\genNsharp$-oldest edge]
  \label{lem:unique-oldest}
  Let $(\tria_0,\frc)$ be an $N{+}1$-colored initial triangulation and let $T \in \mtree$. Then for every edge $e \in \edges(T) \setminus \bse(T)$ we have 
\begin{equation*} 
  \genNsharp(e) > \genNsharp(\bse(T)).
\end{equation*} 
\end{lemma}
\begin{proof}
By design, the $\genNsharp$ generation of edges in $\tria_0$ equals the $\gensharp$ generation in its virtual extension $\tria_0^+$. 
Let $T\in \mtree$ be a subsimplex of a virtual extension $T^+$ resulting from bisections of an $N$-simplex in $\tria_0^+$ such that $\bse(T) = \bse(T^+)$. According to \cite[Lem.~3.12]{DieningStornTscherpel23} the $\gensharp$ generation of the bisection edge $\bse(T^+)$ is strictly smaller than the $\gensharp$ generation of all other edges of $T^+$. 
Since all edges in $T$ are also edges in $T^+$, this property extends to the $\genNsharp$ generation of edges $T$ and so concludes the proof. 
\end{proof}
Lemma~\ref{lem:unique-oldest} leads to the following adjusted version of \cite[Cor.~4.9]{Stevenson08}. 
\begin{corollary}[Monotonicity]\label{cor:Monoton}
Let $(\tria_0,\frc)$ be $N{+}1$-colored.
Suppose $T_0,T_1\in \mtree$ are simplices with bisection edge $\bse(T_0)$ contained in $T_1$, i.e.~$\bse(T_0) \in \edges( T_1)$. Then one of the following alternatives holds:
  \begin{enumerate}
  \item $\bse(T_0) = \bse(T_1)$ and $\genNsharp(T_0) = \genNsharp(T_1)$, \label{itm:refChain1}
  \item $\bse(T_0) \neq \bse(T_1)$ and $\genNsharp(T_0) > \genNsharp(T_1)$.\label{itm:refChain2}
  \end{enumerate}
\end{corollary}
\begin{proof}
 Since $\genNsharp(T_j) = \genNsharp(\bse(T_j))$ for $j=0,1$, the statement is an immediate consequence of Lemma~\ref{lem:unique-oldest}.
 \end{proof}
A consequence of the existence of refinement chains and the previous corollary is the following result.
\begin{corollary}[Limited level increase]\label{cor:levelIncrease}
Let $\tria \in \BisecT$ with $N{+}1$-colored $\tria_0$. Suppose $T$ is a newly created simplex by the refinement \Refine($\tria,M$) with $M\in \tria$ in the sense that $T \in \Refine(\tria,M) \setminus \tria$. Then we have
\begin{align*}
\level(T) \leq \level(M) + 1.
\end{align*} 
\end{corollary}
\begin{proof}
Let $T\in \Refine(\tria,M) \setminus \tria$. Lemma~\ref{lem:RefChainsNew} states the existence of a refinement chain connecting $M$ with the parent $P\in \tria$ of $T$.  Corollary~\ref{cor:Monoton} shows
\begin{align*}
\genNsharp(P) \leq \genNsharp(M)\qquad\text{and thus}\qquad \level(P) \leq \level(M).
\end{align*}
The corollary then follows from the estimate in \eqref{eq:levelIncrease}.
\end{proof}
With these preliminary considerations we are able to verify the  basic properties of~$\BisecT$ stated in Theorem~\ref{thm:basic-properties}.
\begin{proof}[Proof of Theorem~\ref{thm:basic-properties}]
Let $\tria_0$ be an $N{+}1$-colored initial triangulation.

\textit{Proof of \ref{itm:basic-colored-terminate}}.
Lemma~\ref{lem:RefChainsNew} shows that each simplex that is flagged for refinement is connected to $T$ by a refinement chain. Due to Corollary~\ref{cor:Monoton} the $\genNsharp$-generation of the simplices in this chain decreases strictly monotonically. This proves that the length of these chains is bounded. Moreover, each bisection increases the $\genNsharp$ generation of the new simplices. When the $\genNsharp$ generation exceeds $\genNsharp(T)$, the simplex cannot be bisected anymore according to  Lemma~\ref{lem:RefChainsNew} and Corollary~\ref{cor:Monoton}. This proves that after a finite number of bisections the routine in Algorithm~\ref{algo:closure-recursive} terminates.

\textit{Proof of \ref{itm:basic-colored-shape}}.
According to Theorem~\ref{thm:EquiRefinements} the bisection routine on each simplex equals Maubach's bisection routine in Algorithm~\ref{algo:maubach}. 
The existence of maximal $n!n2^{n-2}$ classes of similar simplices in $T$ resulting from successive applications of this routine to an initial simplex $T_0\in \tria_0$  is known, see \cite[Thm.~4.5]{AMP.2000}. The estimate in \ref{itm:basic-colored-shape} is proven in the appendix.

\textit{Proof of \ref{itm:basic-colored-uniform}}. 
The property in~\ref{itm:basic-colored-uniform} results from the property that uniform refinements of the virtual (colored) extension $\tria_0^+$ of $\tria_0$ are conforming, see Theorem~\ref{thm:basic-properties}\ref{itm:basic-colored-uniform} with $N=n$ which has been proven in~\cite{Maubach95,Traxler97,Stevenson08}. Since $N$ refinements in $\tria_0^+$ correspond to $n$  refinements in $\tria_0$, this yields the property in~\ref{itm:basic-colored-uniform}.

\textit{Proof of \ref{itm:basic-colored-lattice}}.
The proof of~\ref{itm:basic-colored-lattice} is done as in case of colored initial triangulations~$\tria_0$ in \cite{DieningKreuzerStevenson16,DieningStornTscherpel23} and \cite[Cor.~4.18 in Sec.~4.4.2]{Gehring}.
\end{proof}
An immediate consequence of Theorem~\ref{thm:basic-properties} is the following observation.
\begin{corollary}[Level and diameter]\label{cor:GensharpToSize}
Let $\tria\in \BisecT$ with $N{+}1$-colored initial triangulation $\tria_0$.
There exist constants $0 < c_{sh}\leq C_{sh}< \infty$ with
\begin{align*}
c_{sh} 2^{-\level(T)} \leq \textup{diam}(T) \leq C_{sh} 2^{-\level(T)} \qquad\text{for all }T\in \tria.
\end{align*}
The ratio $C_{sh}/c_{sh} \lesssim 1$ depends on the dimension $n$, the shape regularity of $\tria_0$, and the quasi-uniformity $\max_{T,K\in \tria_0} |T|/|K| \eqsim 1$, but not on $N$. 
\end{corollary}
\begin{proof}
Let $\tria\in \BisecT$ with $N{+}1$-colored initial triangulation $\tria_0$. 
Similar considerations as in Theorem~\ref{thm:basic-properties}~\ref{itm:basic-colored-uniform} show that the first of each $n$ consecutive bisections increases the level of a simplex,
leading for any descendant $T\in \tria$ of a simplex $T_0 \in \tria_0$ to
\begin{align*}
2^{-\level(T)n} |T_0| \leq  |T| \leq 2^{n-1-\level(T)n } |T_0|.
\end{align*}
Combining this inequality with the shape regularity in Theorem~\ref{thm:basic-properties}~\ref{itm:basic-colored-shape} and the quasi-uniformity $\max_{T,K\in \tria_0} |T|/|K| \eqsim 1$ concludes the proof.
\end{proof}
\begin{lemma}[Distance]\label{lem:distance of simplices}
Let $M \in \tria\in \BisecT$ with $N{+}1$-colored $\tria_0$ and recall the constant $C_{sh}$ from Corollary~\ref{cor:GensharpToSize}. Any new simplex $T\in\BisectionClosure(\tria,M)\setminus \tria$ satisfies
\begin{align*}
\dist(T,M)\coloneqq \inf_{x'\in T,x\in M}\lvert x'-x\rvert \leq 4 C_{sh} N 2^{- \level(T)}.
\end{align*}
\end{lemma}
\begin{proof}
Let $M\in \tria\in \BisecT$ and let $T \in \Refine(\tria,M)\setminus \tria$.
Lemma~\ref{lem:RefChainsNew}~\ref{itm:RefChains2} yields the existence a simplicial chain $T_0,\dots,T_J\in \mtree$ with $J\in \mathbb{N}_0$ such that $M = T_0$, $T$ is a child of $T_J$, and
\begin{align*}
\bse(T_j) \in \edges(T_{j+1})\qquad\text{for all }j=0,\dots,J-1.
\end{align*}
In addition we have due to Corollary~\ref{cor:Monoton} the property
\begin{align*}
\genNsharp(T_{j+1}) < \genNsharp(T_j)\qquad\text{for all }j=0,\dots,J-1.
\end{align*}
Since the $\genNsharp$ generation increases at most $N$ times until
the related $\levelNsharp$ level increases and $\levelNsharp (T_j) = \level(T_j) + 1$, we obtain the upper bound
\begin{align*}
\dist(T,M) &\leq \sum_{j=1}^J \textup{diam}(T_j) \leq C_{sh} \sum_{j=1}^J 2^{-\level(T_j)} \leq  C_{sh}N 2^{-\level(T_J)} \sum_{j=0}^\infty 2^{-j}.
\end{align*}
Using the identity $\sum_{j=0}^\infty 2^{-j} = 2$ and the property 
$\level(T) \leq \level(T_J) + 1$ of the child $T$ concludes the proof.
\end{proof}
For any $\tria\in \BisecT$ with subset $\mathcal{M}\subset \tria$ we denote by $\BisectionClosure(\tria,\mathcal{M})$ the triangulation obtained by refining all simplices in $\mathcal{M}$ in the sense that 
\begin{align*}
\BisectionClosure(\tria,\mathcal{M}) \coloneqq \bigvee_{M\in \mathcal{M}} \BisectionClosure(\tria,M).
\end{align*}
\begin{proof}[Proof of Theorem~\ref{thm:closuregcolored}]
For $N=n$ this result has been shown in \cite[Thm.~2.4]{BinevDahmenDeVore04} and \cite[Thm.~6.1]{Stevenson08}. Our proof modifies the arguments presented therein.

Let us abbreviate $\mathcal M\coloneqq \bigcup_{\ell=0}^{L-1}\mathcal M_\ell$ where $\mathcal{M}_\ell$ denotes the sets of marked simplices as stated in the theorem and $L\in \mathbb{N}$.
Let $E\coloneqq 5C_{sh}N$ and $F \coloneqq C_{sh}+4C_{sh}N$ with constant $C_{sh}$ from Corollary~\ref{cor:GensharpToSize}. We define the neighborhood of any $T \in \tria$ as 
\begin{align*}
\neig(T)\coloneqq \left\lbrace T'\in\mtree \colon \dist(T,T')\leq E 2^{-\level(T)}\right\rbrace.
\end{align*}
We further introduce 
the function $\lambda \colon \mtree \times \mathcal M\to \mathbb R$
\begin{align*}
\lambda(T,M) \coloneqq 
\begin{cases}
F 2^{\level(T)-\level(M)}&\text{if }M \in \neig(T)\text{ and }\level(T) \leq \level(M) + 1,\\
0				&\text{otherwise.}
\end{cases}
\end{align*}
We claim that there exist constants $0 < c$ and $C < \infty$ such that
\begin{align}\label{eq:ProofPart}
\begin{aligned}
\sum_{T\in \tria_L\setminus \tria_0} \lambda(T,M) & \leq C &&\quad \text{for all }M\in \mathcal{M},\\
\sum_{M\in \mathcal M} \lambda(T,M) & \geq c &&\quad \text{for all }T\in \tria_L \setminus \tria_0.
\end{aligned}
\end{align} 
These two estimates yield
\begin{align*}
\#\tria_L-\#\tria_0\leq \#(\tria_L\setminus\tria_0) = \sum_{T\in\tria_L \setminus\tria_0} 1 \leq \frac{1}{c} \sum_{T\in\tria_L\setminus\tria_0}\sum_{ M\in \mathcal M}\lambda(T,M)\leq \frac{C}{c}\,  \#\mathcal M.
\end{align*}
Since this yields the theorem with $\CBDV = C/c$, it remains to verify \eqref{eq:ProofPart}.

Let $M \in \mathcal{M}$ and let $k \leq \level(M) + 1$.
We denote the ball with center $x\in \mathbb{R}^n$ and radius $r>0$ by $B(x,r)\subset \mathbb{R}^n$. 
Corollary~\ref{cor:GensharpToSize} shows that any simplex $T\in \mtree$ with $\level(T) = k$ and $M\in \neig(T)$ is by definition for any $y\in M$ a subset of
\begin{align*}
B_{M,k} \coloneqq \left\lbrace x\in \overline{\Omega} \colon \dist(x,M) \leq (E+C_{sh})2^{-k}\right\rbrace\subset B\big(y,C_{sh} 2^{1-k}+(E+C_{sh})2^{-k}\big).
\end{align*}
The volume of each $T\in \mtree$ with $\level(T) = k$ satisfies due to  Corollary~\ref{cor:GensharpToSize} and shape regularity $c_{sh}^n 2^{-nk} \lesssim |T|$. Thus, the number of such elements in $B_{M,k}$ satisfies
\begin{align*}
\# \lbrace T \in \mtree \colon \level(T) = k \text{ and }M\in \neig(T)\rbrace \lesssim \left(\frac{E+3C_{sh}}{c_{sh}}\right)^n.
\end{align*}
Hence, we obtain the first estimate in \eqref{eq:ProofPart} by 
\begin{align*}
\sum_{T \in \mtree} \lambda(T,M) &= \sum_{k=0}^{\level(M) + 1} \sum_{T \in \mtree, \level(T) = k} \lambda(T,M)\\
& \lesssim \left(\frac{E+3C_{sh}}{c_{sh}}\right)^n F \sum_{k=-1}^{\level(M)} 2^{- k } \leq 4F \left(\frac{E+3C_{sh}}{c_{sh}}\right)^n.
\end{align*}
To prove the lower bound in \eqref{eq:ProofPart}, let $T \in \tria_L\setminus \tria_0$. Moreover, let $T_0,\dots,T_K\in \mtree$ be a sequence of simplices with $T = T_K$ and $T_0 \in \tria_0$ such that $T_j\in \mathcal M_{k_j}$ and $T_{j+1}\in \BisectionClosure(\tria_{k_j},T_j)\setminus \tria_{k_j}$ with $0 \leq k_1 < \dots < k_K < L$.
According to Corollary~\ref{cor:levelIncrease}, the level can increase only by one from one sequence member to the next one in the sense that
\begin{align*}
\level(T_{j+1}) \leq \level(T_j) + 1\qquad\text{for all }j=0,\dots,K-1.
\end{align*}
This and $T_K = T \notin\tria_0$ yield the existence of an index $s\in \lbrace 0,\dots,K-1\rbrace$ with $\level(T_s)\in\{\level(T)-1,\level(T)\}$.
If the sequence $T_0,\dots,T_K$ stays in the neighborhood of $T=T_K$, we have
\begin{align*}
F \leq \lambda(T,T_s) \leq \sum_{j=0}^{K-1} \lambda(T,T_j). 
\end{align*}
Otherwise, let $k\in \lbrace 0,\dots,K\rbrace$ denote the largest index with $T_k \notin \neig(T)$.
The definition of the neighborhood $\neig(T)$, Lemma~\ref{lem:distance of simplices}, and the definition of $\lambda$ show
\begin{align*}
E 2^{-\level(T)}&\leq \dist(T,T_k) \leq \dist(T,T_{K-1})+\sum_{j=k+1}^{K-1} \diameter(T_j)+ \sum_{j=k}^{K-2}\dist(T_{j+1},T_j)\\
			&\leq  4C_{sh}N\,2^{-\level(T)}+C_{sh}\sum_{j=k+1}^{K-1}2^{-\level(T_j)}+ 4C_{sh}N\, \sum_{j=k+1}^{K-1}2^{-\level(T_j)}\\
			&\leq 2^{-\level(T)}\left(4C_{sh}N +\sum_{j=k+1}^{K-1} (C_{sh}+4C_{sh}N)2^{\level(T)-\level(T_j)}\right).
\end{align*}
Due to the definitions $E\coloneqq 5C_{sh}N$ and $F \coloneqq C_{sh}+4C_{sh}N$ this yields
\begin{align*}
c \coloneqq C_{sh} N \leq \sum_{j=0}^{K-1} \lambda(T,T_j).
\end{align*} 
Finally, this results in $\CBDV=C/c\lesssim 4F(E+3C_{sh})^n/(C_{sh}Nc_{sh}^n)\lesssim N^n$.
\end{proof}
%
%
\begin{remark}[Freudenthal's triangulation]
There are meshes like Freudenthal's triangulation consisting of translations of Kuhn cubes that allow for optimal shape regularity and closure estimates.
However, the coloring needed to obtain the corresponding tagged simplices requires many colors, see Figure~\ref{fig:what-does-not-work}. 
\begin{figure}[ht!]
\begin{tikzpicture}
\draw (0,0) grid (5,1);
\foreach \x in {0,...,4} \draw (\x,0)--++(1,1);
\foreach \x in {1,...,5} \draw[dotted] (\x,0)--++(-1,1);
\draw (0,0) node [below]{$N$};
\foreach \x in {0,...,4} \draw (\x+1,0) node [below] {$\x$};
\foreach \x in {0,...,5}	\draw (\x,1) node [above] {$\x$};
\end{tikzpicture}
\caption{
A triangulation which can be tagged satisfying the initial conditions by Binev--Dahmen--DeVore and Stevenson with good shape regularity and closure estimate, but its coloring needs many colors for obtaining the new edges illustrated by the dotted lines.
}
\label{fig:what-does-not-work}
\end{figure}
\end{remark}
%
%
\section{Numerical Experiments}\label{sec:simpl-numer-exper}
We conclude this paper with numerical experiments illustrating the performance of the Maubach routine in Algorithm~\ref{algo:closure-recursive} and~\ref{algo:maubach} with initialization as in Algorithm~\ref{algo:coloring}.
\subsection{Experiment 1 (Properties of the algorithm)}
Our first numerical experiment investigates the properties of our bisection routine for initial triangulations $\tria_0$ and related domains $\Omega$ displayed in Figure~\ref{fig:NetgenMeshes} from the software package Netgen \cite{Schoberl97}.
\begin{figure}[ht!]
\hbox{\begin{minipage}{.2\textwidth}
\includegraphics[scale=.4]{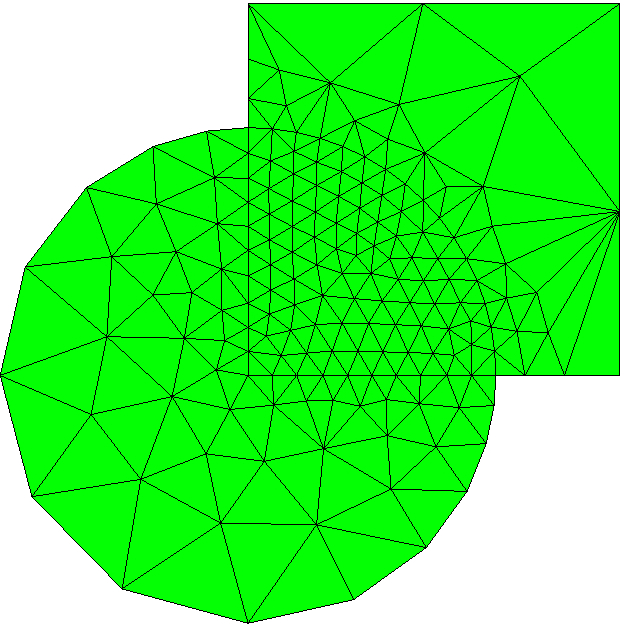}
\end{minipage}
\begin{minipage}{.2\textwidth}
\includegraphics[scale=.4]{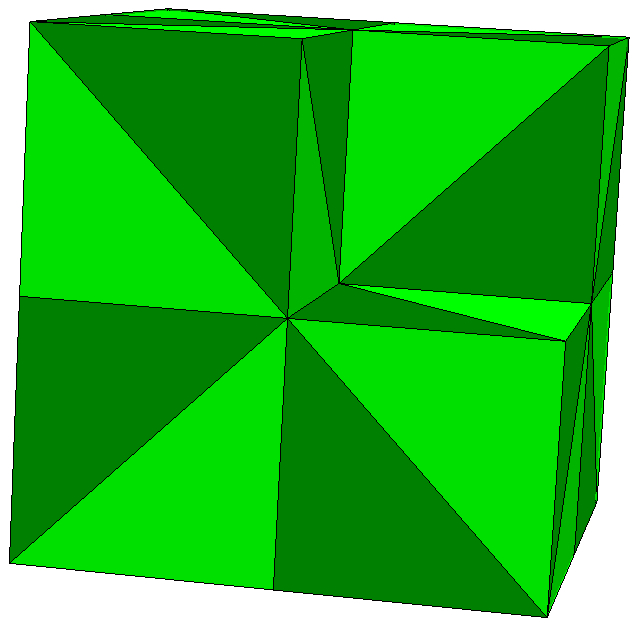}
\end{minipage}
\hspace*{-.4cm}
\begin{minipage}{.2\textwidth}
\includegraphics[scale=.45]{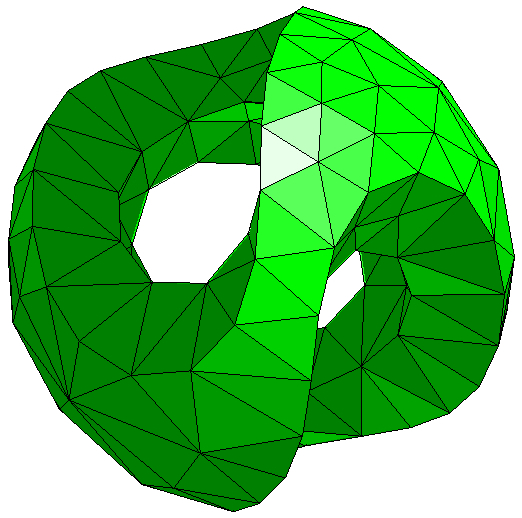}
\end{minipage}
\hspace*{-.5cm}
\begin{minipage}{.2\textwidth}
\includegraphics[scale=.51]{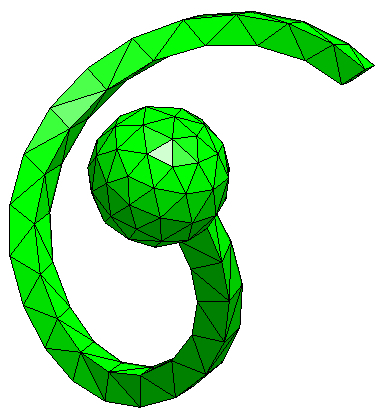}
\end{minipage}
\hspace*{-1.1cm}
\begin{minipage}{.2\textwidth}
\includegraphics[scale=.55]{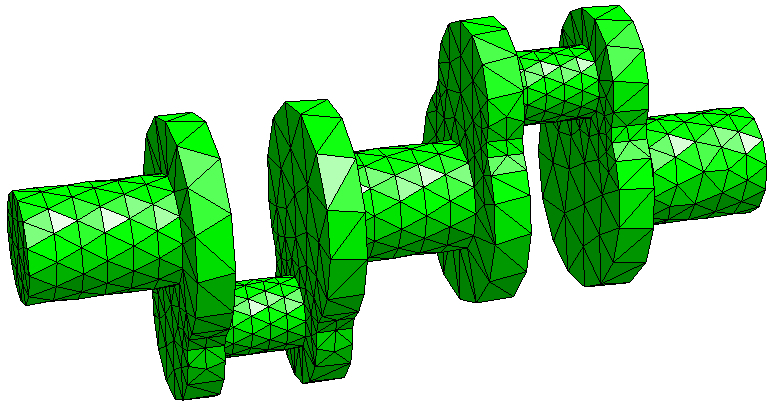}
\end{minipage}}
\caption{Meshes from left to right: 2dMesh, Fichera, Sculpture, Extrusion, Shaft. 
}\label{fig:NetgenMeshes}
\end{figure}
We run the AFEM loop as described in \cite{Stevenson07} with bulk parameter $\Theta = 0.3$ to approximate the Poisson model problem $-\Delta u = 1$ in $\Omega$ with homogeneous Dirichlet boundary condition $u = 0$ on $\partial \Omega$ by the Galerkin finite element method with quadratic Lagrange elements. We stop the AFEM loop when the number of degrees of freedom exceeds $10^5$. 
We apply the bisection routine in Algorithm~\ref{algo:maubach} with closure in Algorithm~\ref{algo:closure-recursive} and display the number of colors $N$ obtained by the initialization with Algorithm~\ref{algo:coloring} in Table~\ref{tab:table}.
Let $\tria_L$ with $L\in \mathbb{N}$ denote the finest mesh and let $\mathcal{M}_\ell$ denote the set of marked elements on the $\ell$-th mesh $\tria_\ell$ with $\ell=0,\dots,L-1$ obtained by the adaptive loop.
Table~\ref{tab:table} further contains the ratio of the shape regularities \eqref{eq:DefShapeReg} and a lower bound for the Binev--Dahmen--DeVore constant in Theorem~\ref{thm:closuregcolored} 
\begin{align*}
\frac{\gamma(\tria_L)}{\gamma(\tria_0)} \coloneqq \frac{\max_{T \in \tria_{L}} \gamma(T) }{\max_{T_0 \in \tria_0} \gamma(T_0)}\qquad\text{and}\qquad \CBDV^{lb} \coloneqq \frac{ \# \tria_L - \# \tria_0}{\sum_{\ell = 0}^{L-1} \# \mathcal{M}_\ell}. 
\end{align*}
For the Fichera corner domain, which consists of seven Kuhn cubes, we further compare the results with the values obtained by a manual coloring that is in agreement with the coloring of the Kuhn cube, cf.~\cite[Sec.~ 4.1]{DieningStornTscherpel23}; the values with coloring obtained by Algorithm~\ref{algo:coloring} and with manual coloring are displayed in the row ``Fichera alg.'' and ``Fichera man.'', respectively.
\begin{table}[ht!]
\begin{tabular}{l|l|l|l}
Initial Mesh & Colors $N$ & $\gamma(\tria_L)/\gamma(\tria_0)$ & $\CBDV^{lb}$\\ \hline 
2dMesh & 5 & 3.14 & 1.76 \\
Fichera alg. & 3 & 2.86 & 3.76 \\
Fichera man. & 3 & 1.06 & 3.30 \\
Sculpture & 7 & 3.91 & 4.66 \\
Extrusion & 6 & 1.94 & 4.19 \\
Shaft & 8 & 3.77& 5.10
\end{tabular}
\ \\[.5em]
\caption{Results of the computations in Experiment 1 with initial partitions displayed in Figure~\ref{fig:NetgenMeshes}.}\label{tab:table}
\end{table}%
The results in Table~\ref{tab:table} show that the bisection routine degrades the shape regularity by a factor between about three to four. 
The colors are bounded by eight, but do not seem to have a big influence on the Binev--Dahmen--DeVore constant. The latter is bounded in all computations by a factor of about two in 2D and five in 3D.
While the Binev--Dahmen--DeVore constant in the manually colored Fichera corner domain is similar to the one obtained by Algorithm~\ref{algo:coloring}, its shape regularity behaves much better. This motivates the use of additional information for the coloring algorithm, resulting for example from the mesh generation routine.
\subsection{Experiment 2 (Comparison)}\label{subsec:Exp2}
\begin{figure}[ht!]
\begin{tikzpicture}
\begin{axis}[
clip=false,
width=.5\textwidth,
height=.45\textwidth,
xmode = log,
ymode = log,
cycle multi list={\nextlist MyColors},
scale = {1},
xlabel={ndof},
clip = true,
legend cell align=left,
legend style={legend columns=1,legend pos= north east,font=\fontsize{7}{5}\selectfont}
]
	\addplot table [x=ndof,y=EnergyError] {Experiments/MC_IrrRegN_5.txt};
	\pgfplotsset{cycle list shift=1}
	\addplot table [x=ndof,y=EnergyError] {Experiments/FEniCS_IrrReg.txt};
	\addplot table [x=ndof,y=EnergyError] {Experiments/Netgen_IrrReg.txt};
	\legend{{Maubach alg.},{FEniCS},{Netgen}};
\end{axis}
\end{tikzpicture}
\begin{tikzpicture}
\begin{axis}[
clip=false,
width=.5\textwidth,
height=.45\textwidth,
xmode = log,
ymode = log,
cycle multi list={\nextlist MyColors},
scale = {1},
xlabel={ndof},
clip = true,
legend cell align=left,
legend style={legend columns=1,legend pos= north east,font=\fontsize{7}{5}\selectfont}
]
	\addplot table [x=ndof,y=EnergyError] {Experiments/MC_Reg.txt};
	\addplot table [x=ndof,y=EnergyError] {Experiments/MC_Reg_val.txt};
	\addplot table [x=ndof,y=EnergyError] {Experiments/FEniCS_Reg.txt};
	\addplot table [x=ndof,y=EnergyError] {Experiments/Netgen_Reg.txt};
	\legend{{Maubach man.}, {Maubach alg.},{FEniCS},{Netgen}};
\end{axis}
\end{tikzpicture}
\caption{Convergence history of the energy error \eqref{eq:EnergyError} in Experiment 2 with the result for the built-in mesh $\tria_0$ in Netgen (left) and the result for $\tria_0$ consisting of Kuhn cubes (right).} \label{fig:convHist}
\end{figure}%
Our second experiment applies the same AFEM loop as the first experiment to two initial partitions of the Fichera corner domain. 
The first partition is the built-in partition in Netgen and does in particular not consist of a union of Kuhn cubes. The second initial partition is the one displayed in Figure~\ref{fig:NetgenMeshes}. We use the manual coloring for the second initial partition and Algorithm~\ref{algo:coloring} (\emph{Maubach man.}~and \emph{Maubach alg.}~in Figure~\ref{fig:convHist}, respectively) for both partitions, leading to the numbers of colors $N=5$ for the first and $N=3$ for the second partition.
We further apply the same AFEM loop but with the realization of the refinement routine \cite{Rivara84,Rivara91} in FEniCS \cite{LoggMardalWells12} and of the refinement routine \cite{AMP.2000} in Netgen \cite{Schoberl97}.
Figure~\ref{fig:convHist} displays the resulting convergence history plots for the squared energy error of the Galerkin approximation $u_h$ which reads, with exact solution $u$,
\begin{align}\label{eq:EnergyError}
\int_\Omega |\nabla (u - u_h)|^2 \, \mathrm{d}x.  
\end{align}
All refinement routines lead to the same optimal rate of convergence.
For the built-in mesh of Netgen the errors displayed in Figure~\ref{fig:convHist} (left) do not differ significantly. 
However, Figure~\ref{fig:convHist} (right) shows that the coloring obtained by Algorithm~\ref{algo:coloring} results in an about 5.3 times larger (squared) error than the manually colored initial partition and in an about 2.2 times larger error than the refinement routines of Netgen and FEniCS for the second initial partition $\tria_0$ and number of degrees of freedom $\textup{ndof} \approx 1.2 \times 10^5$. 
This stresses the importance of including additional information in the coloring approach as for example done in Figure~\ref{fig:what-does-not-work}. Including  such additional information is an open issue motivating further research.
\appendix
\section{Shape regularity}
It is well known in the literature that the shape regularity of simplices resulting from successive application of the bisection routine in Algorithm~\ref{algo:maubach} to an initial simplex is bounded
due to the existence of at most $n!\, n\, 2^{n-2}$ classes of similar simplices for each $T_0 \in \tria_0$, see \cite[Thm.~4.5]{AMP.2000}.
However,  the authors have not seen a bound for the shape regularity in terms of the initial shape regularity yet.
The aim of this appendix is to derive such an upper bound for the shape regularity $\gamma(T)=D(T)/d(T)$ defined in \eqref{eq:DefShapeReg}, see \cite{BrandtsKorotovKrivzek09,BrandtsKorotovKrivzek11} for alternative equivalent definitions. In particular, we verify the following theorem which coincides with the second statement of Theorem~\ref{thm:basic-properties}\ref{itm:basic-colored-shape}.
\begin{theorem}[Shape regularity]\label{thm:shape regularity main result}
A descendant $T$ obtained by successive applications of Algorithm~\ref{algo:maubach} to some initial tagged $n$-simplex $T_0$ satisfies
\begin{align*}
\gamma(T)\leq 2n(n + \sqrt{2}-1)\, \gamma(T_0).
\end{align*}
\end{theorem}
The proof of this theorem uses the following notation. For any given simplex $T$, let $b(T)$ is the largest closed ball included in $T$, $B(T)$ the smallest ball including $T$, and $d(T)$ and $D(T)$ their respective diameters.
\begin{lemma}[Kuhn simplex]
  \label{lem:shape-regularity-kuhn}
  Let $T_0$ be a Kuhn $n$-simplex defined with permutation $\pi\colon \lbrace 1,\dots,n\rbrace \to  \lbrace 1,\dots,n\rbrace$ as tagged $n$-simplex by 
\begin{equation*}  
  T_0 = [0,e_{\pi(1)},e_{\pi(1)} + e_{\pi(2)},\dots,e_{\pi(1)} + \dots + e_{\pi(n)}]_n.  
\end{equation*}
  Then for every descendant $T$ of $T_0$ resulting from Algorithm~\ref{algo:maubach} we have
  \begin{align*}
    \gamma(T) \leq 2 \gamma(T_0).
  \end{align*}
\end{lemma}
\begin{proof}
 Let $T_{k+1}$ denote a child of~$T_k$ for all $k\in \mathbb{N}_0$. Then $T_{k+n}$ is similar to~$T_k$ with scaling factor~$1/2$. Thus, it suffices to prove the claim for $T_1,\dots, T_{n-1}$. 
The diameters $D(T_j)$ and $d(T_j)$ decrease monotonically as $j$ increases. Combining this with $D(T_n)= D(T_0)/2$ and $d(T_n) = d(T_0)/2$ results in
  \begin{align*}
    \gamma(T_j)
    &= \frac{D(T_j)}{d(T_j)}
    \leq \frac{D(T_0)}{d(T_n)} 
    \leq \frac{D(T_0)}{d(T_0)/2} = 2 \gamma(T_0).\qedhere
  \end{align*}
\end{proof}
In the following we exploit the fact that we can transform a Kuhn simplex into any simplex by an affine mapping $F\colon T \to \widehat{T}, x \mapsto Ax+b$ with matrix $A\in \mathbb{R}^{n\times n}$ and vector $b\in \mathbb{R}^n$. Let $\abs{A}$ and $\abs{A^{-1}}$ denote the spectral norm of~$A$ and its inverse, respectively. Moreover, let $\mathcal{F}(T)$ denote the set of hyperfaces ($(n{-}1)$-dimensional subsimplices) of a simplex $T$ and let $h_f \coloneqq h_f(T) \coloneqq \sup_{x\in T} \distance(x,f)$ denote the height of $T$ corresponding to the hyperface $f\in \mathcal{F}(T)$. 
We denote the width of $T$, which equals the minimal height of $T$, by
\begin{align*}
w(T) \coloneqq \min_{f\in \mathcal{F}(T)} h_f.
\end{align*}
\begin{lemma}[Transformation]\label{lem:condition}
Let $F\colon \mathbb R^n\rightarrow \mathbb R^n, x\mapsto Ax+b$ be a bijective affine map and $T\subset \mathbb R^n$ be a simplex. Then there holds
\begin{enumerate}
\item $d(T) \leq \abs{A^{-1}}\, d(F(T)) \leq \diameter(T)$,\label{itm:condition-r}
\item $w(T) \leq D(F(T))/|A| \leq D(T)$, \label{itm:condition-R}
\item $w(T)/\diameter(T) \leq \gamma(F(T))/(|A|\,|A^{-1}|) \leq \gamma(T)$. \label{itm:condition-gamma}
\end{enumerate}
\end{lemma}
\begin{proof}
Let $T$ be a simplex with largest inscribed ball $b(T)$ and smallest ball $B(T)$ containing $T$, that is, $d(T) = \textup{diam}(b(T))$ and $D(T) = \textup{diam}(B(T))$.

\textit{Proof of \ref{itm:condition-r}}.
The ellipsoid $e\coloneqq F (b(T)) \subset F(T)$ includes a ball with diame\-ter $d(T)/\abs{A^{-1}}$, showing the first inequality in \ref{itm:condition-r}. 
The inverse mapping $F^{-1}$ maps any diameter of $b(F(T))$ to a line segment included in $T$, hence the image is shorter than $\diameter (T)$.
Since there exists such a diameter which is mapped to a line segment of length $\lvert A^{-1}\rvert\, d(F(T))$, this shows the second inequality in \ref{itm:condition-r}.

\textit{Proof of \ref{itm:condition-R}}.
The ellipsoid $F(B(T))\supset F(T)$ is included in the ball with the same center and the radius $|A|\, D(T)$, which shows the second inequality of \ref{itm:condition-R}.
To show the first inequality, we use the ellipsoid $E\coloneqq F^{-1}B(F(T))$. Note that the minimal height of a simplex is generalized for arbitrary sets $M\subset \setR^n$ by the width $w(M)$: the minimal distance of a pair of parallel hyperplanes which includes $M$. This function is monotone with respect to inclusion. Additionally, it equals the length of the minor axis for an ellipsoid. Therefore, $T\subset E$ implies $w(T) \leq 2\, \inf_{x\in \partial E}\abs{x-\operatorname{mid} E}$. This property and $\operatorname{mid} (B(F(T))) = F(\operatorname{mid} (E))$ lead to 
\begin{align*}
\abs A&=\sup_{y\in \mathbb{R}^n\setminus \lbrace 0\rbrace} \frac{\abs{Ay}}{\abs{y}} =\sup_{x\in \mathbb{R}^n\setminus \lbrace \operatorname{mid} E \rbrace} \frac{\abs{F(x) - F(\operatorname{mid} E)}}{\abs{x - \operatorname{mid} E}}\\
& 
=\sup_{x\in \partial E} \frac{\abs{F(x)-\operatorname{mid} B(F(T))}}{\abs{x-\operatorname{mid} E}} \leq \frac{D(F(T))}{2\, \inf_{x\in \partial E}\abs{x-\operatorname{mid} E}}
\leq \frac{D(F(T))}{w(T)}.
\end{align*}
This verifies the first inequality of \ref{itm:condition-R}.

\textit{Proof of \ref{itm:condition-gamma}}.
Dividing inequality~\ref{itm:condition-R} by~\ref{itm:condition-r} yields \ref{itm:condition-gamma}.
\end{proof}
With the properties of transformed simplices stated in the previous lemma we obtain the following. 
\begin{lemma}[Bound via Kuhn simplex]
  \label{lem:shape-regularity}
  Let $T_0$ denote a tagged simplex and let $\widehat{T}$ denote a Kuhn simplex. Then any descendant $T$ of $T_0$ satisfies
  \begin{align*}
    \gamma(T) &\leq 2 \gamma(\widehat{T})\frac{\diameter \widehat T}{w(\widehat T)} \gamma(T_0)\qquad\text{for all } T \in \tria \in  \Bisec(\set{T_0}).
  \end{align*}
\end{lemma}
\begin{proof}
Let $F$ denote the mapping from the Kuhn simplex~$\widehat{T}_0$ to $T_0=F(\widehat{T}_0)$ which maps also a descendant $\widehat{T}$ of $\widehat{T}_0$ to $T=\widehat{T}$.
Lemma~\ref{lem:condition} and Lemma~\ref{lem:shape-regularity-kuhn} lead to
  \begin{align*}
    \gamma(F(\widehat{T})) &\leq \gamma(\widehat{T}) \abs{A} \abs{A^{-1}}
              \leq 2\, \gamma(\widehat{T}_0) \abs{A} \abs{A^{-1}}
             \leq 2\, \gamma(\widehat{T}_0)\frac{\diameter(\widehat T_0)}{w(\widehat T_0)} \gamma (F(\widehat{T}_0)).\qedhere
  \end{align*}
\end{proof}

It remains to compute the diameter, minimal height, and shape regularity for the Kuhn simplex to obtain an upper bound with the estimate in Lemma~\ref{lem:shape-regularity}.
\begin{lemma}[Parameters for the Kuhn simplex]\label{lem:parameters for Kuhn}
A Kuhn simplex $\widehat T \subset \mathbb{R}^n$ satisfies
\begin{enumerate}
\item\label{it:diam-R} $\diameter(\widehat T)=D(\widehat T) =\sqrt{n}$,
\item $w(\widehat T) = 1/\sqrt2$ for $n\geq 2$,
\item $1/d(\widehat{T}) = 1 + (n-1)/\sqrt{2}$,
\item $ \gamma(\widehat{T}) = \sqrt{n}\, (1 + (n-1)/\sqrt{2})$,
\item $2\gamma(\widehat{T}) \diameter (\widehat T)/w(\widehat T) =2n\, (n+\sqrt{2}-1)$.
\end{enumerate}
\end{lemma}

\begin{proof}
Let $\widehat{T}=[v_0,\dots,v_n]$ with $v_k=\sum_{i=1}^k e_i$ denote a Kuhn simplex. 
Its longest edge $[v_0,v_n]$ is a diameter of the ball $B\left(1/2\sum_{i=1}^n e_i,\sqrt{n}/2\right) = B(\widehat{T})$, implying \ref{it:diam-R}.

\textit{Step 1 (Formula for $d(T)$)}.
Let $\lvert T\rvert_n$ and $\lvert f\rvert_{n-1}$ denote the $n$- and $(n{-}1)$-dimensional volumes of an $n$-simplex $T$ and its hyperfaces $f\in \mathcal{F}(T)$.
A partition of $T$ by the insphere center into $n+1$ simplices with height $d(T)/2$ shows 
\begin{align}
|T|_n=\frac{d(T)}{2n}\sum_{f\in \mathcal{F}(T)}|f|_{n-1}.
\end{align}
With the formula $|T|_n= h_f |f|_{n-1}/n$ for all $f\in \mathcal{F}(T)$ this yields the identity
  \begin{align}\label{eq:1DivRho}
    \frac{1}{d(T)} &= \frac 12 \sum_{f\in \mathcal{F}(T)} \frac{1}{h_f}.
  \end{align}

\textit{Step 2 (Heights in $\widehat{T}$)}.
The edge $H_0\coloneqq [v_0,v_1]$ of the Kuhn simplex $\widehat{T}$ is the height on the hyperface $[v_1,\dots,v_n]$ and likewise the edge $H_n \coloneqq [v_{n-1},v_n]$ on the hyperface $[v_0,\dots,v_{n-1}]$. These two heights have length one.
For $i=1,\dots,n-1$, the line segment $H_i \coloneqq [v_i,(v_{i-1}+v_{i+1})/2] = v_i+[0,(e_{i+1}-e_i)/2]$ is perpendicular to the edges 
$[v_{j-1},v_j] = v_{j-1}+[0,e_j]$ for all $j\notin \{i,i+1\}$.
Moreover, $H_i$ is perpendicular
to the edge $[v_{i-1},v_{i+1}] = v_{i-1}+[0,e_{i}+e_{i+1}]$. 
Hence, $H_i$ is the height on the hyperface $[v_0,\dots,v_n ]\setminus [ v_i ]$ opposite to $v_i$ and its length is $1/\sqrt2$.
Using this observations in \eqref{eq:1DivRho} leads to the remaining formulas.
\end{proof}
\begin{proof}[Proof of Theorem \ref{thm:shape regularity main result}]
Lemma \ref{lem:shape-regularity} and Lemma \ref{lem:parameters for Kuhn} combine to Theorem \ref{thm:shape regularity main result}.
\end{proof}
\printbibliography 
\end{document}